\documentclass[11pt]{amsart}

\usepackage[letterpaper,margin=1.1in]{geometry}

\usepackage{amsmath,amssymb,mathtools,mathrsfs}
\usepackage{array}
\usepackage{microtype}
\usepackage{tikz}
\usepackage{float}
\usetikzlibrary{arrows.meta,calc,positioning,fit}
\usepackage[hidelinks]{hyperref}
\hypersetup{pdftitle={Cluster Algebras for Bosonic Plethysm},%
  pdfauthor={Yelin Fan and Jiarui Fei},%
  pdfsubject={Cluster structures, theta bases, and bosonic plethysm},%
  pdfkeywords={plethysm, cluster algebras, theta bases, invariant theory, polyhedral cones}}
\usepackage{quivers_plethysm}

\newtheorem{theorem}{Theorem}[section]
\newtheorem{proposition}[theorem]{Proposition}
\newtheorem{corollary}[theorem]{Corollary}
\newtheorem{lemma}[theorem]{Lemma}
\newtheorem*{theoremA}{Theorem A}
\newtheorem*{theoremB}{Theorem B}
\newtheorem*{theoremC}{Theorem C}
\theoremstyle{definition}

\theoremstyle{remark}
\newtheorem{remark}[theorem]{Remark}
\newtheorem{example}[theorem]{Example}

\newcommand{\kk}{\Bbbk}
\newcommand{\GL}{\operatorname{GL}}
\newcommand{\Sym}{\operatorname{Sym}}
\newcommand{\SymMat}{\operatorname{SymMat}}
\newcommand{\htp}{\operatorname{ht}}
\newcommand{\cP}{\mathcal P}
\newcommand{\cR}{\mathcal R}
\newcommand{\cT}{\mathcal T}
\newcommand{\cU}{\mathcal U}
\newcommand{\frakS}{\mathfrak S}
\newcommand{\Mat}{\operatorname{Mat}}
\newcommand{\SL}{\operatorname{SL}}
\newcommand{\wt}{\operatorname{wt}}
\newcommand{\Frac}{\operatorname{Frac}}
\newcommand{\Rep}{\operatorname{Rep}}
\newcommand{\SI}{\operatorname{SI}}
\newcolumntype{L}[1]{>{\raggedright\arraybackslash}p{#1}}

\title{Cluster Algebras for Bosonic Plethysm}
\author{Yelin Fan}
\address{School of Mathematical Sciences, Shanghai Jiao Tong University, Shanghai, China}
\email{sjtu\_stanfan@sjtu.edu.cn}
\author{Jiarui Fei}
\address{School of Mathematical Sciences, Shanghai Jiao Tong University, Shanghai, China}
\email{jiarui@sjtu.edu.cn}
\thanks{The authors were supported in part by the National Natural Science Foundation of China (Nos.~12131015 and 12571038).}
\thanks{An earlier version of this work was completed in 2022.}
\date{}

\begin{document}
\raggedbottom
\begin{abstract}
Let $\Bbbk$ be an algebraically closed field of characteristic zero, let
$V=\Bbbk^\ell$ and $W=\Bbbk^m$, and set
\[ \mathcal R_{\ell,m}=\operatorname{Sym}(\operatorname{Sym}^2V\otimes W)^{U_V}. \]
We construct an explicit skew-symmetrizable seed $\Sigma_{\ell,m}$ by
restricting and folding the determinantal seed for the flagged $m$-arrow
Kronecker quiver.  For every $\ell,m\ge2$, we have
\[ \mathcal R_{\ell,m}=\mathcal U(\Sigma_{\ell,m}), \]
with polynomial frozen coefficients, and $\Sigma_{\ell,m}$ admits a
reddening sequence.  The theta functions that extend across the frozen boundary are indexed
by the integral points of a rational polyhedral cone $\mathscr C_{\ell,m}$.
Its weight fibers count the multigraded highest-weight multiplicities of
$\mathcal R_{\ell,m}$, and the Jacobi--Trudi identity expresses symmetric-square
plethysm coefficients as finite alternating sums of these counts.
For odd $\ell$ and even $m$, optimized seeds for the frozen divisors give
an explicit finite system of inequalities for $\mathscr C_{\ell,m}$;
degree-zero faces give polyhedral models for odd $m$.
\end{abstract}
\maketitle

\section{Introduction}
\label{sec:introduction}

Plethysm concerns the decomposition of a composite Schur functor
\[ \mathbf S_\mu(\mathbf S_\nu V) \cong \bigoplus_\lambda a_{\mu,\nu}^{\lambda}\,\mathbf S_\lambda V. \]
Here $\mu$ is the outer partition, $\nu$ is the inner partition, and
$\lambda$ is the resulting $\GL(V)$-highest weight.  The multiplicities
$a_{\mu,\nu}^{\lambda}$ are among the basic structure constants of polynomial
representation theory, but their general behavior remains poorly understood;
see, for example, \cite[Chapter~I]{Macdonald}.
This paper treats the case $\nu=(2)$.  Thus the inner functor is the
symmetric square, and we refer to this as the bosonic case.  The exterior
square gives the corresponding fermionic problem, studied in \cite{FFP}.

Fix $V=\kk^\ell$ and $W=\kk^m$, so that
$\htp(\lambda)\leq\ell$ and $\htp(\mu)\leq m$.  The double highest-weight algebra
\[ \Sym(\Sym^2V\otimes W)^{U_V\times U_W} \]
has the plethysm coefficients as its joint weight multiplicities.  It is useful, however, to retain the full $\GL(W)$-action
and consider instead
\begin{equation}
 \cR_{\ell,m}:=\Sym(\Sym^2V\otimes W)^{U_V}.
 \label{eq:intro-main-ring}
\end{equation}
After a basis of $W$ has been chosen, the algebra is graded by
$\alpha=(\alpha_1,\ldots,\alpha_m)\in\mathbb Z_{\ge0}^m$.  Its
$T_V$-weight space of weight $\lambda$ has dimension
\[ b_{\alpha,(2)}^\lambda =\dim\operatorname{Hom}_{\GL(V)}\!\left( \mathbf S_\lambda V, \bigotimes_{r=1}^m\Sym^{\alpha_r}(\Sym^2V) \right). \]
These multiplicities are nonnegative and come from an algebra that
retains the $\GL(W)$-action, unlike the algebra of $U_W$-invariants.  The Jacobi--Trudi identity recovers
plethysm from them:
\begin{equation}
 a_{\mu,(2)}^\lambda
 =\sum_{\sigma\in\frakS_m}\operatorname{sgn}(\sigma)
   b_{\alpha^\sigma(\mu),(2)}^\lambda,
 \qquad
 \alpha^\sigma(\mu)_i=\mu_i-i+\sigma(i).
 \label{eq:intro-JT}
\end{equation}
A term is understood to vanish when some component of
$\alpha^\sigma(\mu)$ is negative.  We use a cluster structure on \eqref{eq:intro-main-ring} to describe
these weight spaces.

\subsection{Main results}

Write $X_{\ell,m}=\SymMat_\ell^m$.  The algebra
$\cR_{\ell,m}$ is the invariant ring $\kk[X_{\ell,m}]^{U_\ell}$ for the
simultaneous congruence action
\[ u\cdot(C_1,\ldots,C_m) =(u^{-T}C_1u^{-1},\ldots,u^{-T}C_mu^{-1}). \]
For each $\ell,m\ge2$ we define a seed $\Sigma_{\ell,m}$ whose extended
cluster consists of the determinants $\Delta_r=\det C_r$ and a family of
chamber determinants $z_{i,j}^{(n)}$.  Its frozen set has $m+\ell-1$
vertices, and its mutable set has $\frac{1}{2}{(m-1)(\ell-1)(\ell+2)}$
vertices.  The exchange relations are given explicitly in
Theorem~\ref{thm:folded-relations}.
The seed $\Sigma_{4,3}$ is shown in Figure~\ref{fig:seed-43}.

\begin{theoremA}
For every $\ell,m\ge2$, the exchange matrix of $\Sigma_{\ell,m}$ has full
row rank and a skew-symmetrizable principal part, and
\[ \cR_{\ell,m}=\mathcal U(\Sigma_{\ell,m}) \]
with polynomial frozen coefficients.
\end{theoremA}

\begin{theoremB}
For every $\ell,m\ge2$, the seed $\Sigma_{\ell,m}$ admits a reddening
sequence.
\end{theoremB}

Theorem~B is proved in Theorem~\ref{thm:reddening}.  For odd $m$, the mutable
exchange matrix is reduced, by mutations and source--sink deletions, to a
string-diagram matrix of type $A_{\ell-1}$.  The even case then follows
because $B_{\ell,m}^{\mathrm{uf}}$ is a principal submatrix of
$B_{\ell,m+1}^{\mathrm{uf}}$.

Let $M_{\ell,m}$ be the character lattice of the initial cluster torus, and
let $D_f$ be the prime divisor associated with a frozen variable $z_f$.
For integral $g$, the values $\operatorname{ord}_{D_f}(\vartheta_g)$ extend
to integral piecewise-linear functions $\nu_f$ on $M_{\ell,m,\mathbb R}$.
Put
\[ \mathscr C_{\ell,m} =\{g\in M_{\ell,m,\mathbb R}: \nu_f(g)\ge0 \text{ for every frozen index }f\}. \]
The boundary inequalities define a rational polyhedral cone.  Let
$\mathsf W_{\ell,m}$ denote the weight map of the initial seed and set
\[ \mathscr P_{\ell,m}(\lambda;\alpha) =\{g\in\mathscr C_{\ell,m}: \mathsf W_{\ell,m}g=(\lambda;\alpha)\}. \]

\begin{theoremC}
For every $\ell,m\ge2$, the family
\[ \{\vartheta_g:g\in\mathscr C_{\ell,m}\cap M_{\ell,m}\} \]
is a basis of $\cR_{\ell,m}$.  Moreover,
\[ b_{\alpha,(2)}^\lambda =\bigl|\mathscr P_{\ell,m}(\lambda;\alpha)\cap M_{\ell,m}\bigr|, \]
and hence
\[ a_{\mu,(2)}^\lambda =\sum_{\sigma\in\frakS_m}\operatorname{sgn}(\sigma) \bigl|\mathscr P_{\ell,m} (\lambda;\alpha^\sigma(\mu))\cap M_{\ell,m}\bigr|. \]
\end{theoremC}

The basis statement is Theorem~\ref{thm:theta-basis}, and the two formulas are
Corollary~\ref{cor:polyhedral-b} and Theorem~\ref{thm:polyhedral-plethysm}.
Thus the coefficients $b_{\alpha,(2)}^\lambda$ have a lattice-point
interpretation, while the plethysm coefficients are obtained from these counts
by the Jacobi--Trudi formula \eqref{eq:intro-JT}.  We do not obtain a positive
rule for $a_{\mu,(2)}^\lambda$ itself.

When $\ell$ is odd and $m$ is even, every frozen boundary divisor has an
optimized seed.  Pulling the corresponding chiral-dual monomial back to the
initial chart gives a finite Laurent polynomial, and its tropicalization
gives the boundary inequality.  This produces an integral matrix
$H_{\ell,m}$ with
\[ \mathscr C_{\ell,m} =\{g:H_{\ell,m}g\ge0\}; \]
see Proposition~\ref{prop:optimized-boundary-seeds} and
Corollary~\ref{cor:boundary-cone-optimized-seeds}.  The cone and theta basis exist for all parities.  The construction of
optimized seeds requires $\ell$ odd and $m$ even.

Proposition~\ref{prop:successive-matrix-face} relates the models for different
numbers of matrices by taking faces.  For $2\le r<m$, the simultaneous
zero-degree locus for $C_{r+1},\ldots,C_m$ is an exposed face of $\mathscr
C_{\ell,m}$, and the theta functions indexed by its integral points form a
basis of $\cR_{\ell,r}$.  In particular, when $\ell$ is odd, the explicit
inequalities for $2n$ matrices restrict to a polyhedral model for $2n-1$
matrices.

\subsection{The folded determinantal seed}

The construction begins with the two-sided matrix invariant ring
\[ \cT_{\ell,m} =\kk[\Mat_\ell^m]^{U_\ell^{\mathrm L}\times U_\ell^{\mathrm R}}. \]
The cluster structure on the semi-invariant ring of the flagged Kronecker
quiver in \cite{FeiKroneckerII} transfers to $\cT_{\ell,m}$.  Its cluster
functions are determinants of matrices formed from alternating products of
the $A_r$ and $A_r^{-1}$; determinant factors clear the apparent
denominators, so these are polynomial invariants.  Matrix transposition
interchanges the two flag arms and acts on this seed.

Restriction from arbitrary matrices to symmetric matrices sends the left and
right chamber functions in each transpose pair to the same determinant. This
identification of functions does not by itself determine an exchange matrix.
One must prove that transposition is admissible and then sum the columns over
each transpose orbit.  Fixed chambers have orbit size one and the remaining
chambers have orbit size two; their interaction produces the valuations
$(1,2)$ and $(2,1)$ in the folded quiver.  Consecutive triangles are glued
along alternating sides: the triangles with indices $n-1$ and $n$ share the
side $i+j=\ell$ when $n$ is odd, and the side $i=0$ when $n$ is even.  The
orbit-sum construction gives a skew-symmetrizer equal to the orbit-size
function.

Algebraic independence is proved on a Gauss open set.  The fixed torus of the
two-sided cluster chart is a dense open subset of
$T_\ell\times\SymMat_\ell^{m-1}$, and its coordinate functions are precisely
the folded initial family.  The maximal-rank upper-bound theorem gives the
first inclusion.  The invariant ring is factorial; irreducibility of the
initial functions and coprimality with their one-step mutations imply
$\mathcal U(\Sigma_{\ell,m})\subseteq\cR_{\ell,m}$.  For the reverse
inclusion, the $LDL^T$ decomposition gives generators after localization at
the leading principal minors of $C_1$.  Restriction from the two-sided matrix
algebra is surjective on this open set, and the same argument applies to
$C_2$.  Intersecting the two localizations removes all denominators and yields
Theorem~A.  In particular, the central determinants are not inverted in the
resulting algebra.

The reddening argument uses the disk quivers and flip sequences of Fock and
Goncharov \cite{FockGoncharov}.  A transpose-fixed half-diamond diagonal is
removed directly in the folded matrix.  Between two such removals, the two
flip regions are disjoint and exchanged by transposition, so their mutation
sequences descend by orbit mutation.  This reduces the odd-$m$ case to a
string diagram; the result for even $m$ follows by passage to a principal
submatrix.

Finally, the reddening sequence and full row rank of the exchange matrix
imply the full Fock--Goncharov condition in the sense of Gross--Hacking--Keel--Kontsevich
\cite{GHKK}.  The open cluster variety therefore has a theta basis indexed by
all integral tropical points.  The polynomial frozen coefficients define a
partial compactification whose boundary consists of the frozen prime divisors.
Valuative independence and theta reciprocity \cite{CMMM} show that the theta
functions regular along this boundary form a basis.  This proves Theorem~C.

\subsection{Relation to earlier work}

Cluster algebras and upper cluster algebras were introduced in \cite{FZ1,BFZ};
the coefficient formalism used here is compatible with \cite{FZ4}.  Cluster
structures on semi-invariant rings of triple flags and flagged Kronecker
quivers were constructed in \cite{Fsemi,FeiKroneckerII}.  Sink--source
extensions and frozen amalgamations give another construction of the
simply-laced hive quivers for complete $m$-tuple flags
\cite{FeiWeymanExtension}. This construction describes the triangles of the
unfolded seed, but does not give its symmetric restriction or its folding by
transposition.  Admissible orbit folding is part of the standard theory of
non-simply-laced cluster algebras; see \cite{DupontFolding}.  Recent work of
Ye constructs folded cluster structures on $SL_n/SO_n$ and on the affine space
of a single symmetric matrix \cite{YeSymmetric}.  Those varieties and seeds
are distinct from the highest-weight invariant algebras of tuples of symmetric
matrices considered here.

The representation-theoretic identities in
Section~\ref{sec:plethysm-highest-weight} are standard consequences of the
Cauchy decomposition, highest-weight theory, and the Jacobi--Trudi identity;
proofs are included to fix the gradings and sign conventions.  The two-sided
transfer and the unfolded seed come from \cite{Grosshans,FeiKroneckerII}, the
upper-bound argument uses \cite{GSV}, the reddening argument uses
\cite{FockGoncharov,CaoString,CaoLi}, and the theta-basis argument uses
\cite{GHKK,CMMM}.

To the authors' knowledge, the following results are new: the symmetric
restriction and explicit orbit-sum folded seed for all $\ell,m$; the folded
exchange relations, full-rank statement, and equality with $\cR_{\ell,m}$; the
reddening construction for every $\ell,m\ge2$; the theta basis of the
invariant ring and the resulting plethysm formulas; and the optimized boundary
seeds when $\ell$ is odd and $m$ is even.  None of these conclusions is a
formal specialization of the flagged Kronecker construction: functions related
by transposition have the same restriction, fixed orbits give valued arrows,
and both the exchange matrix and the equality with the upper cluster algebra
require separate proofs.

\subsection{Organization}

Section~\ref{sec:plethysm-highest-weight} relates plethysm to the multigraded
highest-weight multiplicities of $\cR_{\ell,m}$.
Section~\ref{sec:gaussian-localization} describes the symmetric Gauss chart
and its polynomial coordinate lifts.  Section~\ref{sec:matrix-invariants}
transfers the flagged Kronecker seed to two-sided matrix invariants and
restricts the cluster functions to symmetric matrices.
Section~\ref{sec:folding} constructs the folded exchange matrix and proves the
exchange relations.  Section~\ref{sec:upper-cluster-equality} proves
Theorem~A, Section~\ref{sec:reddening} proves Theorem~B, and
Section~\ref{sec:polyhedral-plethysm} proves Theorem~C, the optimized boundary
results, and the degree-zero face construction.  The appendices identify the
matrix functions under transfer, prove the remaining coprimality cases, and
give the folded half-diamond calculation and the induction for boundary
optimization.

\section{Plethysm and highest-weight invariants}
\label{sec:plethysm-highest-weight}

The identities in this section are standard consequences of the Cauchy
formula, highest-weight theory, and the Jacobi--Trudi identity; see
\cite[Chapter~I]{Macdonald}.  We include the proofs in order to fix the
multigrading and sign conventions used later.

Throughout the paper, $\kk$ is an algebraically closed field of
characteristic zero.  If $E$ is a finite-dimensional $\kk$-vector space, we
write
\[ G_E=\GL(E). \]
After choosing an ordered basis of $E$, let $B_E=T_EU_E$ be the standard
upper-triangular Borel subgroup, with diagonal torus $T_E$ and upper
unitriangular subgroup $U_E$.

A partition is a weakly decreasing sequence of nonnegative integers with
finite support.  For a partition $\rho$, write $|\rho|$ for its size and
$\htp(\rho)$ for the number of its nonzero parts.  In plethysm, $\mu$ is the
outer partition, $\nu$ the inner partition, and $\lambda$ the resulting
$G_V$-highest weight.  The symbol $\alpha=(\alpha_1,\ldots,\alpha_m)\in\mathbb
Z_{\geq0}^m$ denotes a weak composition.  We write $\mathbf S_\rho$ for the
Schur functor associated with $\rho$; in particular, $\mathbf
S_{(d)}E=\Sym^d(E)$.

Let $V$ and $W$ be vector spaces of dimensions $\ell$ and $m$, respectively,
and let $\nu$ be a partition with $\htp(\nu)\leq \ell$.  We consider the
polynomial algebra
\begin{equation}
 \cP_\nu(V,W)
 :=\Sym\bigl(\mathbf S_\nu V\otimes W\bigr).
 \label{eq:ambient-polynomial-algebra}
\end{equation}
Equivalently,
\[ \cP_\nu(V,W) =\kk\bigl[\mathbf S_\nu(V^\vee)\otimes W^\vee\bigr], \]
so \eqref{eq:ambient-polynomial-algebra} may be viewed either as a symmetric
algebra or as the coordinate ring of the dual representation.  The group
$G_V\times G_W$ acts naturally on this algebra.  Its $U_V$-invariant
subalgebra
\[ \cR_\nu(V,W):=\cP_\nu(V,W)^{U_V} \]
will be called the \emph{highest-weight algebra} associated with
$\mathbf S_\nu V\otimes W$.

\subsection{Plethysm coefficients as double highest-weight multiplicities}

For partitions $\lambda$ and $\mu$, with
$\htp(\lambda)\leq\ell$ and $\htp(\mu)\leq m$, the plethysm coefficient
$a_{\mu,\nu}^{\lambda}$ is defined by
\[ \mathbf S_\mu(\mathbf S_\nu V) \cong \bigoplus_{\rho} a_{\mu,\nu}^{\rho}\,\mathbf S_\rho V. \]
Under these height assumptions, the multiplicity is independent of the
choice of $V$.

For a $T_V\times T_W$-module $M$ and dominant weights $\lambda$ of $G_V$ and
$\mu$ of $G_W$, let $M_{\lambda;\mu}$ denote the corresponding joint weight
space.

The Cauchy formula and highest-weight theory give
\[ a_{\mu,\nu}^{\lambda}=\dim_\kk\left(\cP_\nu(V,W)^{U_V\times U_W}\right)_{\lambda;\mu}; \]
see \cite[Chapter~I]{Macdonald}.

The double invariant algebra records plethysm coefficients directly.  We
instead retain the full $G_W$-action and study $\cR_\nu(V,W)$.

\subsection{Multigraded multiplicities}

Fix a basis $w_1,\ldots,w_m$ of $W$.  The resulting decomposition
$W=\bigoplus_{r=1}^m\kk w_r$ gives $\cP_\nu(V,W)$ and
$\cR_\nu(V,W)$ a $\mathbb Z_{\geq0}^m$-grading.  For
$\alpha=(\alpha_1,\ldots,\alpha_m)$, the corresponding homogeneous component
is naturally isomorphic to
\begin{equation}
 \cP_\nu(V,W)_\alpha
 \cong
 \bigotimes_{r=1}^m
 \Sym^{\alpha_r}(\mathbf S_\nu V).
 \label{eq:multidegree-piece}
\end{equation}
For a partition $\lambda$ with $\htp(\lambda)\leq\ell$, define
\begin{equation}
 b_{\alpha,\nu}^{\lambda}
 :=
 \dim_\kk
 \left(\cR_\nu(V,W)_\alpha\right)_\lambda.
 \label{eq:b-definition}
\end{equation}
We refer to these numbers as the \emph{multigraded highest-weight
multiplicities}.  By convention, $b_{\alpha,\nu}^{\lambda}=0$ if any component of
$\alpha$ is negative.

By \eqref{eq:multidegree-piece} and highest-weight theory,
\begin{equation}
 b_{\alpha,\nu}^{\lambda}
 =\dim_\kk\operatorname{Hom}_{G_V}\!\left(\mathbf S_\lambda V,
   \bigotimes_{r=1}^m\Sym^{\alpha_r}(\mathbf S_\nu V)\right).
 \label{eq:b-as-multiplicity}
\end{equation}

\begin{remark}
Although $\alpha$ is an ordered $T_W$-weight, the multiplicity
$b_{\alpha,\nu}^{\lambda}$ is unchanged when the entries of $\alpha$ are
permuted.
\end{remark}

\subsection{Recovering plethysm coefficients}

The multiplicities \eqref{eq:b-definition} determine the plethysm
coefficients by a finite alternating sum.
For a partition $\mu$ with $\htp(\mu)\leq m$, pad $\mu$ with
zeros and write
$\mu=(\mu_1,\ldots,\mu_m).$
For $\sigma\in\frakS_m$, define an integral $m$-tuple
\begin{equation}
 \alpha^\sigma(\mu)_i
 :=\mu_i-i+\sigma(i),
 \qquad 1\leq i\leq m.
 \label{eq:shifted-composition}
\end{equation}
The equality $|\alpha^\sigma(\mu)|=|\mu|$ holds, but
$\alpha^\sigma(\mu)$ need not be a weak composition.

\begin{theorem}
\label{thm:alternating-plethysm}
For partitions $\lambda,\mu,\nu$ satisfying
$\htp(\lambda)\leq\ell$, $\htp(\mu)\leq m$, and
$\htp(\nu)\leq\ell$, one has
\begin{equation}
 a_{\mu,\nu}^{\lambda}
 =
 \sum_{\sigma\in\frakS_m}
 \operatorname{sgn}(\sigma)\,
 b_{\alpha^\sigma(\mu),\nu}^{\lambda}.
 \label{eq:alternating-plethysm}
\end{equation}
Terms for which \eqref{eq:shifted-composition} has a negative component are
understood to be zero.
\end{theorem}

\begin{proof}
Apply the plethystic ring homomorphism $f\mapsto f[s_\nu]$ to the
Jacobi--Trudi identity
\[ s_\mu=\det\bigl(h_{\mu_i-i+j}\bigr)_{1\le i,j\le m} \]
and expand the determinant.  By \eqref{eq:b-as-multiplicity},
the coefficient of $s_\lambda$ in the term indexed by $\sigma$ is
$b_{\alpha^\sigma(\mu),\nu}^{\lambda}$.  Comparing coefficients proves
\eqref{eq:alternating-plethysm}.
\end{proof}

\begin{example}
When $m=2$ and $\mu=(\mu_1,\mu_2)$, Theorem
\ref{thm:alternating-plethysm} reads
\[ a_{\mu,\nu}^{\lambda} = b_{(\mu_1,\mu_2),\nu}^{\lambda} - b_{(\mu_1+1,\mu_2-1),\nu}^{\lambda}. \]
\end{example}

\subsection{The symmetric-square specialization}

We call plethysm with inner partition $(2)$, equivalently with the symmetric-square
functor $\Sym^2$, the \emph{bosonic} case.  The exterior-square analogue,
which gives fermionic plethysm, is treated in \cite{FFP}.  For the remainder
of the paper we specialize to $\nu=(2)$ and set
\[ V=\kk^\ell, \qquad W=\kk^m. \]
We abbreviate
\begin{equation}
 \cR_{\ell,m}
 :=\cR_{(2)}(V,W)
 =\Sym\bigl(\Sym^2V\otimes W\bigr)^{U_V}.
 \label{eq:main-ring-abstract}
\end{equation}
To obtain the matrix realization, let
\[ \SymMat_\ell :=\{C\in\operatorname{Mat}_{\ell\times\ell}(\kk):C^T=C\}, \qquad X_{\ell,m}:=\SymMat_\ell^m. \]
After choosing the standard bases of $V$ and $W$, the algebra in
\eqref{eq:main-ring-abstract} identifies with
\begin{equation}
 \cR_{\ell,m}=\kk[X_{\ell,m}]^{U_\ell},
 \label{eq:main-ring-matrices}
\end{equation}
where $U_\ell$ is the upper unitriangular subgroup of $\GL_\ell$ and
\[ u\cdot(C_1,\ldots,C_m) = \bigl(u^{-T}C_1u^{-1},\ldots,u^{-T}C_mu^{-1}\bigr), \qquad u^{-T}:=(u^{-1})^T. \]
For $0\le i\le\ell$, set
\begin{equation}
 h_i:=\det C_1[1,\ldots,i\mid1,\ldots,i],
 \qquad h_0:=1,
 \label{eq:hi-definition}
\end{equation}
and put
\[ H_1:=\prod_{i=1}^{\ell}h_i. \]
Theorem~\ref{thm:alternating-plethysm} recovers plethysm coefficients for
inner partition $(2)$ from the multigraded weight multiplicities of
\eqref{eq:main-ring-matrices}.

\section{Gauss-decomposition coordinates on the invariant ring}
\label{sec:gaussian-localization}

On the principal open set defined by the leading principal minors, the
$U_\ell$-action has a cross-section described by the $LDL^T$ decomposition.

\subsection{The Gauss chart}
\label{subsec:gaussian-slice}

Let
\[ X_{\ell,m}^{\mathrm G} :=\{(C_1,\ldots,C_m)\in X_{\ell,m}:H_1(C)\ne0\}. \]
Every $C_1$ occurring in this open set has a unique $LDL^T$
decomposition
\begin{equation}
 C_1=u^TDu,
 \qquad u\in U_\ell,
 \qquad D\in T_\ell.
 \label{eq:symmetric-gaussian-decomposition}
\end{equation}
For $C=(C_1,\ldots,C_m)\in X_{\ell,m}^{\mathrm G}$, set
\[ S_r:=u^{-T}C_ru^{-1}\in\SymMat_\ell, \qquad 2\le r\le m. \]

\begin{proposition}[Gauss chart]
\label{prop:gaussian-slice}
The morphism
\begin{align}
 \Psi:T_\ell\times\SymMat_\ell^{m-1}\times U_\ell
 &\longrightarrow X_{\ell,m}^{\mathrm G},
 \notag\\
 (D,S_2,\ldots,S_m,u)
 &\longmapsto
 (u^TDu,u^TS_2u,\ldots,u^TS_mu)
 \label{eq:gaussian-slice-map}
\end{align}
is an isomorphism.  Under this isomorphism the congruence action of
$w\in U_\ell$ is right translation on the last factor:
\begin{equation}
 w\cdot(D,S_2,\ldots,S_m,u)
 =(D,S_2,\ldots,S_m,uw^{-1}).
 \label{eq:gaussian-action}
\end{equation}
Consequently,
\begin{equation}
 \cR_{\ell,m}[H_1^{-1}]
 \cong\kk[T_\ell\times\SymMat_\ell^{m-1}].
 \label{eq:gaussian-quotient-ring}
\end{equation}
\end{proposition}

\begin{proof}
The decomposition \eqref{eq:symmetric-gaussian-decomposition} is the usual
$LDL^T$ decomposition, written with the upper unitriangular factor
$u=L^T$.  Its diagonal entries are
\begin{equation}
 D_{aa}=\frac{h_a}{h_{a-1}},
 \qquad 1\le a\le\ell.
 \label{eq:gaussian-diagonal-pivots}
\end{equation}
Gaussian elimination expresses $u$ regularly on the principal open
$H_1\ne0$, so the assignment
\[ C\longmapsto(D,S_2,\ldots,S_m,u) \]
is the inverse of \eqref{eq:gaussian-slice-map}.  Formula
\eqref{eq:gaussian-action} follows from the congruence action.  Taking
invariants removes the $U_\ell$-factor, and localization commutes with
invariants because $H_1$ is invariant.  This proves
\eqref{eq:gaussian-quotient-ring}.
\end{proof}

For $2\le r\le m$ and $1\le a\le b\le\ell$, let
\[ \xi_{r;a,b}(D,S_2,\ldots,S_m):=(S_r)_{ab}. \]
Since the $h_i$ are Laurent coordinates on $T_\ell$ by
\eqref{eq:gaussian-diagonal-pivots}, Proposition~\ref{prop:gaussian-slice}
gives
\begin{equation}
 \cR_{\ell,m}[H_1^{-1}]
 =\kk\bigl[
 h_1^{\pm1},\ldots,h_\ell^{\pm1},
 \xi_{r;a,b}:2\le r\le m,
 1\le a\le b\le\ell
 \bigr].
 \label{eq:localized-ring-xi}
\end{equation}

\subsection{Polynomial lifts of the Gauss coordinates}
\label{subsec:gaussian-coordinate-invariants}

The functions $\xi_{r;a,b}$ are regular on the quotient of
$X_{\ell,m}^{\mathrm G}$ but are generally rational on $X_{\ell,m}$.  We
define polynomial lifts of these coordinates.  Put $I_a=\{1,\ldots,a\}$ and, for
$1\le p\le a\le\ell$, define
\[ \gamma_{a,p}(C_1) :=(-1)^{a+p}\det C_1[I_{a-1},I_a\setminus\{p\}], \]
where the determinant of the empty matrix is $1$.  For
$2\le r\le m$ and $1\le a\le b\le\ell$, set
\begin{equation}
 \eta_{r;a,b}
 :=\sum_{p=1}^a\sum_{q=1}^b
 \gamma_{a,p}(C_1)\gamma_{b,q}(C_1)(C_r)_{pq}.
 \label{eq:eta-definition}
\end{equation}

\begin{proposition}
\label{prop:eta-gaussian-coordinate}
Each $\eta_{r;a,b}$ is a polynomial element of $\cR_{\ell,m}$.  On the
Gauss chart,
\begin{equation}
 \eta_{r;a,b}=h_{a-1}h_{b-1}\xi_{r;a,b}.
 \label{eq:eta-slice-formula}
\end{equation}
Consequently,
\begin{equation}
 \cR_{\ell,m}[H_1^{-1}]
 =\kk\bigl[
 h_1^{\pm1},\ldots,h_\ell^{\pm1},
 \eta_{r;a,b}:2\le r\le m,
 1\le a\le b\le\ell
 \bigr].
 \label{eq:localized-ring-eta}
\end{equation}
\end{proposition}

\begin{proof}
Let $C_1=u^TDu$ be \eqref{eq:symmetric-gaussian-decomposition}.  Cramer's
rule gives
\begin{equation}
 (u^{-1})_{p,a}=\frac{\gamma_{a,p}(C_1)}{h_{a-1}},
 \qquad 1\le p\le a.
 \label{eq:u-inverse-cofactor}
\end{equation}
Since $u^{-1}$ is upper triangular,
\[ \xi_{r;a,b} =\sum_{p=1}^a\sum_{q=1}^b (u^{-1})_{p,a}(C_r)_{pq}(u^{-1})_{q,b}. \]
Substituting \eqref{eq:u-inverse-cofactor} proves
\eqref{eq:eta-slice-formula}.  The expression
\eqref{eq:eta-definition} is polynomial and is invariant on the dense open
set $X_{\ell,m}^{\mathrm G}$; hence it is invariant on all of
$X_{\ell,m}$.  Equation \eqref{eq:localized-ring-eta} follows from
\eqref{eq:localized-ring-xi}.
\end{proof}

\section{Two-sided matrix invariants and symmetric restriction}
\label{sec:matrix-invariants}

Assume $\ell,m\ge2$.  The hive-quiver construction for complete triple flags
appears in \cite{Fsemi}; the flagged Kronecker construction used here is
developed in \cite{FeiKroneckerII}.  We formulate the latter in terms of
two-sided matrix invariants and then restrict it to symmetric matrices.  The
transfer is recalled in Appendix~\ref{app:matrix-transfer}.

\subsection{Matrix invariant algebras}
\label{subsec:two-matrix-invariant-algebras}

Retain $X_{\ell,m}=\SymMat_\ell^m$ from Section~\ref{sec:plethysm-highest-weight}, and put
\[ Y_{\ell,m}:=\Mat_\ell^m, \qquad X_{\ell,m}\subseteq Y_{\ell,m}. \]
We use two copies $U_\ell^{\mathrm L}$ and $U_\ell^{\mathrm R}$ of the upper
unitriangular group, acting on $Y_{\ell,m}$ by
\begin{equation}
 (u,v)\cdot(A_1,\ldots,A_m)
 =\bigl(u^{-T}A_1v^{-1},\ldots,u^{-T}A_mv^{-1}\bigr).
 \label{eq:two-sided-action}
\end{equation}
Define the \emph{two-sided matrix invariant algebra}
\[ \cT_{\ell,m} :=\kk[Y_{\ell,m}]^{U_\ell^{\mathrm L}\times U_\ell^{\mathrm R}}. \]
The diagonal embedding $u\mapsto(u,u)$ preserves $X_{\ell,m}$ and induces
the congruence action used in \eqref{eq:main-ring-matrices}.  Restriction
therefore gives a graded homomorphism
\[ \rho:\cT_{\ell,m}\longrightarrow\cR_{\ell,m}. \]
Transposition defines an involution
\[ \tau_Y(A_1,\ldots,A_m)=(A_1^T,\ldots,A_m^T) \]
of $Y_{\ell,m}$.  It exchanges the two unipotent factors in
\eqref{eq:two-sided-action}, hence induces an involution $\tau$ of
$\cT_{\ell,m}$.  Since $X_{\ell,m}$ is the fixed locus of $\tau_Y$,
one has $\rho\circ\tau=\rho.$
For $1\le r\le m$, put
\[ \Delta_r:=\det A_r\in\cT_{\ell,m}; \]
we use the same notation for its restriction $\det C_r$.

\begin{proposition}
\label{prop:basic-rings}
The algebras $\cT_{\ell,m}$ and $\cR_{\ell,m}$ are normal factorial domains,
and
\begin{align*}
 \dim\cT_{\ell,m}
 &=m\ell^2-\ell(\ell-1)=(m-1)\ell^2+\ell,\\
 \dim\cR_{\ell,m}
 &=m\binom{\ell+1}{2}-\binom{\ell}{2}
   =\frac{(m-1)\ell^2+(m+1)\ell}{2}.
\end{align*}
\end{proposition}

\begin{proof}
Invariant subrings of normal domains are normal.  These maximal-unipotent
invariant rings are finitely generated by Grosshans' theorem.  Factoriality
follows from Popov's theorem, because the acting unipotent groups are
connected and have no nontrivial algebraic characters; see
\cite[Theorem~3.17]{PopovVinberg} and \cite{Grosshans}.

For $\cT_{\ell,m}$, the point $(I_\ell,0,\ldots,0)$ has trivial stabilizer:
$u^{-T}v^{-1}=I_\ell$ forces the upper unitriangular matrix $v$ to equal the
lower unitriangular matrix $u^{-T}$, so $u=v=I_\ell$.  For $\cR_{\ell,m}$,
the same point has trivial stabilizer under congruence, since
$u^{-T}u^{-1}=I_\ell$ forces $u=I_\ell$.  The formulas now follow from Rosenlicht's dimension theorem
\cite{Rosenlicht}.
\end{proof}

The torus $T_\ell^{\mathrm L}\times T_\ell^{\mathrm R}\times T_W$ grades
$\cT_{\ell,m}$.  If a homogeneous function has degree
$(\lambda^{\mathrm L},\lambda^{\mathrm R};\alpha)$, then its symmetric
restriction has degree
\begin{equation}
 (\lambda^{\mathrm L}+\lambda^{\mathrm R};\alpha).
 \label{eq:degree-restriction}
\end{equation}

\subsection{The two-sided Gauss chart}
\label{subsec:two-sided-gaussian-slice}

The two-sided action has an analogous Gauss chart.  For the first
matrix in $Y_{\ell,m}$, set
\[ \widetilde h_i:=\det A_1[I_i,I_i], \qquad \widetilde H_1:=\prod_{i=1}^{\ell}\widetilde h_i, \]
and let $Y_{\ell,m}^{\mathrm G}$ be the principal open subset
$\widetilde H_1\ne0$.

\begin{proposition}[two-sided Gauss chart]
\label{prop:two-sided-gaussian-slice}
The morphism
\begin{align}
 \widetilde\Psi:
 T_\ell\times\Mat_\ell^{m-1}\times U_\ell\times U_\ell
 &\longrightarrow Y_{\ell,m}^{\mathrm G},
 \notag\\
 (D,B_2,\ldots,B_m,u,v)
 &\longmapsto
 (u^TDv,u^TB_2v,\ldots,u^TB_mv)
 \label{eq:two-sided-gaussian-map}
\end{align}
is an isomorphism.  The action of $(a,b)\in
U_\ell^{\mathrm L}\times U_\ell^{\mathrm R}$ is right translation on the
last two factors:
\begin{equation}
 (a,b)\cdot(D,B_2,\ldots,B_m,u,v)
 =(D,B_2,\ldots,B_m,ua^{-1},vb^{-1}).
 \label{eq:two-sided-gaussian-action}
\end{equation}
Consequently,
\begin{equation}
 \cT_{\ell,m}[\widetilde H_1^{-1}]
 \cong\kk[T_\ell\times\Mat_\ell^{m-1}].
 \label{eq:two-sided-gaussian-quotient}
\end{equation}
Under this identification, transposition acts by
\begin{equation}
 (D,B_2,\ldots,B_m)
 \longmapsto(D,B_2^T,\ldots,B_m^T).
 \label{eq:transpose-on-two-sided-slice}
\end{equation}
Its fixed locus is therefore
$T_\ell\times\SymMat_\ell^{m-1}$, which is the quotient in
Proposition~\ref{prop:gaussian-slice}.
\end{proposition}

\begin{proof}
The nonvanishing of the leading principal minors is equivalent to the
existence of a unique $LDU$ decomposition
\[ A_1=u^TDv, \qquad u,v\in U_\ell, \qquad D\in T_\ell. \]
The remaining normalized matrices are $B_r=u^{-T}A_rv^{-1}$.  Gaussian
elimination shows that these assignments and their inverses are regular on
$Y_{\ell,m}^{\mathrm G}$, proving
\eqref{eq:two-sided-gaussian-map}.  Formula
\eqref{eq:two-sided-gaussian-action} follows from the two-sided action.
Taking invariants removes the last two factors and proves
\eqref{eq:two-sided-gaussian-quotient}.  Finally,
\[ (u^TDv,u^TB_2v,\ldots,u^TB_mv)^T =(v^TDu,v^TB_2^Tu,\ldots,v^TB_m^Tu), \]
which gives \eqref{eq:transpose-on-two-sided-slice}.  The assertion about
the fixed locus follows.
\end{proof}

\begin{corollary}
\label{cor:localized-restriction-surjective}
Restriction to symmetric matrices induces a surjection
\[ \rho_{H_1}:\cT_{\ell,m}[\widetilde H_1^{-1}] \twoheadrightarrow\cR_{\ell,m}[H_1^{-1}]. \]
\end{corollary}

\begin{proof}
Under the two Gauss-chart identifications, this is the surjective
restriction
\[ \kk[T_\ell\times\Mat_\ell^{m-1}] \longrightarrow\kk[T_\ell\times\SymMat_\ell^{m-1}]. \]
\end{proof}

\subsection{Transfer from the flagged Kronecker model}
\label{subsec:matrix-transfer}

The construction in \cite{FeiKroneckerII} is stated for the semi-invariant
ring of a flagged Kronecker representation space.
Appendix~\ref{app:matrix-transfer} identifies that ring with $\cT_{\ell,m}$
and compares the corresponding cluster functions.

\begin{proposition}\label{prop:matrix-transfer}
Let $\mathsf K_{\ell,m}$ be the $\ell$-flagged $m$-arrow Kronecker quiver
with its standard dimension vector $\boldsymbol\beta_\ell$.  There is a
natural multigraded algebra isomorphism
\[ \SI_{\boldsymbol\beta_\ell}(\mathsf K_{\ell,m}) \xrightarrow{\sim}\cT_{\ell,m}. \]
It sends the determinant of the $r$-th central arrow to $\Delta_r$ and
intertwines reflection of the two flag arms with matrix transposition.
\end{proposition}

The proof is given in Appendix~\ref{app:matrix-transfer}.

\subsection{Alternating words and chamber determinants}
\label{subsec:matrix-chambers}

For $0\le a\le\ell$, let
\[ J_a:\kk^a\hookrightarrow\kk^\ell, \qquad P_a:=J_a^T:\kk^\ell\twoheadrightarrow\kk^a \]
be the standard inclusion and projection.  Empty row or column blocks
are omitted.  If
$I=(r_1,\ldots,r_{2q+1})$ has odd length, define on $(\GL_\ell)^m$
\begin{equation}
 A_I:=A_{r_1}A_{r_2}^{-1}A_{r_3}\cdots
       A_{r_{2q}}^{-1}A_{r_{2q+1}},
 \label{eq:alternating-word}
\end{equation}
and let $\overleftarrow I=(r_{2q+1},\ldots,r_1)$.  Put
\[
 [a]:=(1,2,\ldots,a),
 \qquad
 \mathbf q_n:=
 \begin{cases}
  [n-1],&n\text{ even},\\
  [n],&n\text{ odd},
 \end{cases}
 \qquad 2\le n\le m,
\]
and
\[ \Pi_n:=\prod_{\substack{2\le r<n\\r\text{ even}}}\Delta_r. \]

Let $i,j\ge0$ and $k=i+j$.  On $(\GL_\ell)^m$ set
\begin{equation}
 \phi_{i,j}^{(n),\mathrm L,\circ}(A)
 :=\det
 \begin{bmatrix}
  P_iA_nJ_k\\[1mm]
  P_jA_{\mathbf q_n}J_k
 \end{bmatrix}.
 \label{eq:left-chamber-determinant}
\end{equation}
For $1\le k\le\ell$, define recursively
\begin{equation}
 \phi_{i,j}^{(n),\mathrm L}:=
 \begin{cases}
  \phi_{i,0}^{(n),\mathrm L,\circ},&j=0,\\[1mm]
  \Pi_n\phi_{i,j}^{(n),\mathrm L,\circ},
    &j>0\text{ and }(k<\ell\text{ or }n\text{ even}),\\[1mm]
  \phi_{i,j}^{(n-1),\mathrm L},
    &j>0,\ k=\ell,\ n\text{ odd},
 \end{cases}
 \label{eq:regular-left}
\end{equation}
and set
\begin{equation}
 \phi_{i,j}^{(n),\mathrm R}:=\tau
 \bigl(\phi_{i,j}^{(n),\mathrm L}\bigr).
 \label{eq:regular-right}
\end{equation}
When the middle line of \eqref{eq:regular-left} applies,
\begin{equation}
 \phi_{i,j}^{(n),\mathrm R}(A)
 =\Pi_n\det\left[
  P_kA_nJ_i\ \middle|\ P_kA_{\overleftarrow{\mathbf q_n}}J_j
 \right].
 \label{eq:right-formula}
\end{equation}

\begin{lemma}\label{lem:denominator-clearing}
The functions $\phi_{i,j}^{(n),\mathrm L}$ and
$\phi_{i,j}^{(n),\mathrm R}$ extend to polynomial functions on
$Y_{\ell,m}$.
\end{lemma}

\begin{proof}
Apply Laplace expansion along the top $i$ rows in
$\phi_{i,j}^{(n),\mathrm L,\circ}$, then expand
$\det(A_{\mathbf q_n}[I,J])$ by the Cauchy--Binet formula and apply
Jacobi's complementary-minor identity
\[ \det(B)\det(B^{-1}[I,J]) =(-1)^{\sum I+\sum J}\det(B[J^c,I^c]). \]
Each term contains minors of $A_r^{-1}$ with $r$ even.  Multiplication
by $\Pi_n$ clears their denominators, so
$\Pi_n\phi_{i,j}^{(n),\mathrm L,\circ}$ is polynomial.  Transposition
gives the right chamber function.
\end{proof}

The chamber expressions therefore define polynomial functions on $Y_{\ell,m}$;
no determinant is inverted in the coefficient ring.

\begin{lemma}
\label{lem:chamber-invariance}
Every $\phi_{i,j}^{(n),\mathrm L}$ and
$\phi_{i,j}^{(n),\mathrm R}$ belongs to $\cT_{\ell,m}$.
\end{lemma}

\begin{proof}
For every odd-length word $I$,
\[ ((u,v)\cdot A)_I=u^{-T}A_Iv^{-1}. \]
The chamber matrix is therefore changed by unitriangular row and column
operations.  Its determinant and the factor $\Pi_n$ are invariant.
Transposition gives the right chamber functions.
\end{proof}

\subsection{The flagged Kronecker seed in matrix form}
\label{subsec:flagged-matrix-form}

The determinant formulas above determine the relevant functions up to nonzero
scalars.  We use the compatible normalization from \cite{FeiKroneckerII}.

\begin{theorem}\label{thm:flagged-matrix-form}
There exist nonzero scalars $c_{i,j}^{(n)}$ such that, after setting
\[ x_{i,j}^{(n),\mathrm L}:= c_{i,j}^{(n)}\phi_{i,j}^{(n),\mathrm L}, \qquad x_{i,j}^{(n),\mathrm R}:=\tau \bigl(x_{i,j}^{(n),\mathrm L}\bigr), \]
the determinant functions $\Delta_r$ and the nonredundant chamber functions
$x_{i,j}^{(n),\mathrm L},x_{i,j}^{(n),\mathrm R}$ form the extended cluster
of a seed
\[ \widetilde\Sigma_{\ell,m} =(\widetilde B_{\ell,m},\widetilde{\mathbf x}_{\ell,m}) \quad\text{in }\Frac(\cT_{\ell,m}), \]
whose exchange matrix has full row rank.  Moreover,
\begin{equation}
 \cT_{\ell,m}=\mathcal U(\widetilde\Sigma_{\ell,m}),
 \label{eq:flagged-UCA}
\end{equation}
with frozen variables treated as polynomial coefficients.  The normalization
may be chosen so that
\begin{align}
 x_{i,0}^{(n),\mathrm L}&=x_{i,0}^{(n),\mathrm R}
 =\det(P_iA_nJ_i),
 \notag\\
 x_{0,j}^{(2),\mathrm L}&=x_{0,j}^{(2),\mathrm R}
 =\det(P_jA_1J_j),
 \notag\\
 x_{0,j}^{(n),\epsilon}&=x_{0,j}^{(n-1),\epsilon}
 &&(n\ge4\text{ even}),
 \label{eq:even-identification}\\
 x_{i,\ell-i}^{(n),\epsilon}
 &=x_{i,\ell-i}^{(n-1),\epsilon}
 &&(n\ge3\text{ odd}),
 \label{eq:odd-identification}
\end{align}
where $\epsilon\in\{\mathrm L,\mathrm R\}$.  The redundant corners are
normalized as
\[ x_{\ell,0}^{(n),\epsilon}=\Delta_n, \qquad x_{0,\ell}^{(n),\epsilon} =\prod_{\substack{1\le r<n\\r\text{ odd}}}\Delta_r. \]
Transposition preserves the mutable and frozen subsets and satisfies
$\widetilde b_{\tau u,\tau v}=\widetilde b_{uv}.$
\end{theorem}

\begin{proof}
Via Proposition~\ref{prop:matrix-transfer}, the functions above are the
minimally lifted cluster functions of
\cite[Sections~4--5]{FeiKroneckerII}.  The equality
\eqref{eq:flagged-UCA} is \cite[Theorem~0.1]{FeiKroneckerII}, and full row
rank is proved in \cite[Section~5]{FeiKroneckerII}.  Reflection of the two
flags becomes
matrix transposition, including the determinant coefficients introduced by
the minimal lift.
\end{proof}

We use these normalizations from now on.  The notation $f\doteq g$ means that
$f$ and $g$ differ by a nonzero scalar.  Formulas
\eqref{eq:left-chamber-determinant}--\eqref{eq:right-formula} determine the
normalized functions up to $\doteq$; the exchange relations are normalized
with coefficient $+1$.

\subsection{Symmetric restriction and the folded initial family}
\label{subsec:folded-family}

For a symmetric tuple $C$ and any odd-length sequence $I$,
\begin{equation}
 C_I^T=C_{\overleftarrow I}.
 \label{eq:symmetric-word}
\end{equation}
The two determinant matrices in \eqref{eq:left-chamber-determinant} and
\eqref{eq:right-formula} are therefore transposes.

\begin{proposition}\label{prop:transpose-comparison}
For every admissible $n,i,j$, we have
\begin{equation}
 \rho\bigl(x_{i,j}^{(n),\mathrm L}\bigr)
 =\rho\bigl(x_{i,j}^{(n),\mathrm R}\bigr).
 \label{eq:transpose-comparison}
\end{equation}
\end{proposition}

\begin{proof}
By \eqref{eq:symmetric-word}, the two chamber matrices are transposes on the
symmetric locus.
\end{proof}

Denote the common value in \eqref{eq:transpose-comparison} by
\[ z_{i,j}^{(n)} :=\rho\bigl(x_{i,j}^{(n),\mathrm L}\bigr) =\rho\bigl(x_{i,j}^{(n),\mathrm R}\bigr). \]
Let
\[ \mathscr H_\ell :=\{(i,j)\in\mathbb Z_{\ge0}^2:1\le i+j\le\ell\} \setminus\{(\ell,0),(0,\ell)\}. \]
Define
\begin{equation}
 \mathscr V_{\ell,m}
 :=\left(\bigsqcup_{n=2}^m\{n\}\times\mathscr H_\ell\right)\big/\sim,
 \label{eq:V-index}
\end{equation}
where $\sim$ is generated by
\begin{align*}
 (n;0,j)&\sim(n-1;0,j)
 &&(n\ge4\text{ even}),\\
 (n;i,\ell-i)&\sim(n-1;i,\ell-i)
 &&(n\ge3\text{ odd}).
\end{align*}
For $v=[n;i,j]\in\mathscr V_{\ell,m}$, put $z_v=z_{i,j}^{(n)}$, and set
\[ \mathbf x_{\ell,m} :=\{\Delta_1,\ldots,\Delta_m\} \cup\{z_v:v\in\mathscr V_{\ell,m}\}. \]

The principal minors from \eqref{eq:hi-definition} occur in the folded initial
family as
\[ h_i=z_{0,i}^{(2)},\qquad h_\ell=\Delta_1. \]

We write
\[ \omega_a=(\underbrace{1,\ldots,1}_{a},0,\ldots,0) \quad(0\le a\le\ell), \qquad \omega_0=0, \]
for the fundamental polynomial weights of $\GL_\ell$.  We write
$\varepsilon_1,\ldots,\varepsilon_\ell$ for the standard characters of
$T_\ell$, and let $\delta_1,\ldots,\delta_m$ be the standard basis of
$\mathbb Z^m$.  In particular,
\[ \wt_V(\Delta_r)=2\omega_\ell, \qquad \deg_W(\Delta_r)=\ell\delta_r. \]

For $v\in\mathscr V_{\ell,m}$, let the least index of a representative be
\[ s(v):=\min\{n:(n;i,j)\in v\text{ for some }i,j\}. \]
For a representative $(n;i,j)$, this index is $n-1$ when
$n\ge4$ is even and $i=0$, or when $n\ge3$ is odd, $i+j=\ell$, and
$j>0$; otherwise it is $n$.  Put
$q(v)=\lfloor(s(v)-1)/2\rfloor$ when $j>0$, and $q(v)=0$ when $j=0$.

\begin{proposition}
\label{prop:weights-count}
Let $v=[n;i,j]\in\mathscr V_{\ell,m}$, put $s=s(v)$ and $k=i+j$.  Then
\begin{equation}
 \wt_V(z_v)=\omega_i+\omega_j+\omega_k+2q(v)\omega_\ell.
 \label{eq:V-weight}
\end{equation}
Its $T_W$-multidegree is
\begin{equation}
 \deg_W(z_v)=
 \begin{cases}
  i\delta_s,&j=0,\\[1mm]
  i\delta_s
  +j\displaystyle\sum_{\substack{r<s\\r\text{ odd}}}\delta_r
  +(\ell-j)\displaystyle\sum_{\substack{r<s\\r\text{ even}}}\delta_r,
  &j>0,\ s\text{ even},\\[4mm]
  (i+j)\delta_s
  +j\displaystyle\sum_{\substack{r<s\\r\text{ odd}}}\delta_r
  +(\ell-j)\displaystyle\sum_{\substack{r<s\\r\text{ even}}}\delta_r,
  &j>0,\ s\text{ odd}.
 \end{cases}
 \label{eq:W-degree}
\end{equation}
The cardinality is
\[ |\mathbf x_{\ell,m}| =\frac{(m-1)\ell^2+(m+1)\ell}{2} =\dim\cR_{\ell,m}. \]
\end{proposition}

\begin{proof}
Before restriction, the determinant in \eqref{eq:left-chamber-determinant}
at the least representative $s=s(v)$ uses the first $i$ rows of $A_s$,
the first $j$ rows of $A_{\mathbf q_s}$, and the first $k$ columns of both
matrices.  Each factor in $\Pi_s$
contributes one full left and one full right fundamental weight.  Adding the
left and right weights under \eqref{eq:degree-restriction} gives
\eqref{eq:V-weight}.  Formula \eqref{eq:W-degree} follows by adding the degrees of the
alternating factors and of $\Pi_s$.

Finally,
\[ |\mathscr H_\ell|=\frac{\ell^2+3\ell-4}{2}. \]
There are $m-1$ copies of $\mathscr H_\ell$ and $m-2$ boundary
identifications, each identifying $\ell-1$ pairs.
After adjoining the $m$ determinant functions, one obtains
\[ m+(m-1)|\mathscr H_\ell|-(m-2)(\ell-1) =\frac{(m-1)\ell^2+(m+1)\ell}{2}. \]
The last equality with the dimension follows from
Proposition~\ref{prop:basic-rings}.
\end{proof}

Theorem~\ref{thm:folded-family-transcendence-basis} shows that
$\mathbf x_{\ell,m}$ is a transcendence basis.  For the first two triangles, the chamber functions are, up to
normalization,
\[ z_{i,j}^{(2)}\doteq \det\left[P_{i+j}C_2J_i\ \middle|\ P_{i+j}C_1J_j\right] \]
and, for $j>0$ and $i+j<\ell$,
\[ z_{i,j}^{(3)}\doteq \Delta_2\det\left[ P_{i+j}C_3J_i\ \middle|\ P_{i+j}(C_3C_2^{-1}C_1)J_j \right]. \]
The second formula is polynomial by Lemma~\ref{lem:denominator-clearing}.

\subsection{Algebraic independence of the folded initial family}
\label{subsec:folded-family-independence}

\begin{theorem}
\label{thm:folded-family-transcendence-basis}
The folded initial family $\mathbf x_{\ell,m}$ is a transcendence basis of
$\Frac(\cR_{\ell,m})$.  In particular,
\[ \Frac(\cR_{\ell,m})=\kk(\mathbf x_{\ell,m}). \]
\end{theorem}

\begin{proof}
Let $\widetilde{\mathbb T}_{\ell,m}$ be the initial cluster torus of
\cite{FeiKroneckerII}.  It is
contained in the Gauss open set because the $n=2$ boundary contains
$\widetilde h_1,\ldots,\widetilde h_\ell$.  Transposition exchanges paired
chamber coordinates and fixes the remaining coordinates.  Its fixed
subscheme is therefore a torus with one coordinate for each transpose
orbit.

By Proposition~\ref{prop:two-sided-gaussian-slice}, the fixed locus of the
two-sided Gauss quotient is
\[ T_\ell\times\SymMat_\ell^{m-1} =\operatorname{Spec}\cR_{\ell,m}[H_1^{-1}]. \]
The fixed torus is a nonempty open subset of this irreducible variety, and
its coordinates restrict to $\mathbf x_{\ell,m}$ by
Proposition~\ref{prop:transpose-comparison}.  Hence
$\Frac(\cR_{\ell,m})=\kk(\mathbf x_{\ell,m}). $
Together with $|\mathbf x_{\ell,m}|=\dim\cR_{\ell,m}$, this proves
algebraic independence.
\end{proof}

\section{Folding the matrix seed}
\label{sec:folding}

We fold the matrix seed of Theorem~\ref{thm:flagged-matrix-form} by
transposition.  Theorem~\ref{thm:folded-family-transcendence-basis} has
already identified the folded cluster functions as a transcendence basis. The
exchange matrix and exchange relations are determined below.  We use the
exchange-matrix conventions of \cite{BFZ}, with frozen variables as polynomial
coefficients, and the standard admissible-orbit terminology of
\cite{DupontFolding}.

\subsection{Folding after restriction}
\label{subsec:restriction-folding}

In an extended exchange matrix, the rows are indexed by mutable vertices
and the columns by all vertices.  We call a row \emph{mutable} when its
index is mutable, and a column \emph{frozen} when its index is frozen.
Our exchange convention is
\[ x_ux_u' =\prod_vx_v^{[b_{uv}]_+} +\prod_vx_v^{[-b_{uv}]_+}. \]
Let $\widetilde I$ be the index set of an extended cluster, and let $\tau$ be
an involution preserving its mutable subset $\widetilde M$.  We say that $\tau$ is
\emph{admissible} for $\widetilde B$ if
\begin{align}
 \widetilde b_{\tau u,\tau v}&=\widetilde b_{uv},
 \label{eq:admissible-equivariance}\\
 \widetilde b_{u,\tau u}&=0
 &&(u\text{ mutable}),
 \label{eq:admissible-orbit}\\
 \widetilde b_{uv}\widetilde b_{u,\tau v}&\ge0
 &&(u\text{ mutable}).
 \label{eq:admissible-sign}
\end{align}
For orbits $\bar u,\bar v$, define
\begin{equation}
 b_{\bar u,\bar v}
 :=\sum_{v'\in\bar v}\widetilde b_{u,v'},
 \qquad u\in\bar u.
 \label{eq:orbit-sum}
\end{equation}

\begin{lemma}\label{lem:restriction-folding}
Let $R,S$ be domains and $\rho_0:R\to S$ a homomorphism.  Suppose that
$\widetilde{\mathbf x}$ satisfies the exchange relations of
$\widetilde B$ in $R$, that $\tau$ is admissible, and that
\[ \rho_0(\widetilde x_v)=\rho_0(\widetilde x_{\tau v})=:z_{\bar v}. \]
Then restriction of the relation at $u$ is
\[ z_{\bar u}z_{\bar u}' =\prod_{\bar v}z_{\bar v}^{[b_{\bar u,\bar v}]_+} +\prod_{\bar v}z_{\bar v}^{[-b_{\bar u,\bar v}]_+}, \]
where $z_{\bar u}'=\rho_0(\widetilde x_u')$.  This value is independent of
the representative $u\in\bar u$.
\end{lemma}

\begin{proof}
For a fixed mutable index $u$, condition \eqref{eq:admissible-sign}
implies that the entries $\widetilde b_{u,v'}$, with $v'$ in one orbit,
have the same sign.  Hence
\[ \sum_{v'\in\bar v}[\widetilde b_{u,v'}]_+ =\left[\sum_{v'\in\bar v}\widetilde b_{u,v'}\right]_+, \]
and similarly for negative parts.  Apply $\rho_0$ to the unfolded exchange
relation.  Equivariance shows that the result is independent of the representative.
\end{proof}

Under the transfer isomorphism, the initial exchange relations of
$\widetilde\Sigma_{\ell,m}$ are those of
\cite[Corollaries~5.3--5.4]{FeiKroneckerII}; the variables obtained by one mutation are
regular by Theorem~\ref{thm:flagged-matrix-form}.  The restrictions of these
relations to symmetric matrices are written in
Theorem~\ref{thm:folded-relations}.

\subsection{Admissibility and the folded matrix}
\label{subsec:folded-matrix}

\begin{proposition}
\label{prop:transpose-admissible}
Transposition is admissible for $\widetilde B_{\ell,m}$.
\end{proposition}

\begin{proof}
Transposition exchanges the left and right relations at interior
vertices and identified boundaries, and fixes the relations on the common base.  The determinant coefficients are
fixed and occur in transpose-paired positions, proving
\eqref{eq:admissible-equivariance}.  Distinct members of a nontrivial transpose orbit are never adjacent.  In a relation at an interior vertex or an identified boundary, at most
one member of a nontrivial orbit occurs in either monomial.  On the
common base, both members occur in the same monomial.  Hence \eqref{eq:admissible-orbit} and
\eqref{eq:admissible-sign} hold.
\end{proof}

The proof of the upper cluster algebra equality uses admissibility only at the
initial seed. In Section~\ref{sec:reddening}, fixed diamond diagonals are
removed directly in the folded matrix.  The unfolded construction is used only
for pairs of transpose-related disk flips with disjoint supports;
admissibility is preserved along these sequences.

Introduce determinant indices $\partial_1,\ldots,\partial_m$ and put
$z_{\partial_r}=\Delta_r$.  Let
\[ \mathscr I_{\ell,m} :=\{\partial_1,\ldots,\partial_m\}\sqcup\mathscr V_{\ell,m}. \]
The outer boundary is
\begin{equation*}
\mathscr B_{\ell,m}:=
 \begin{cases}
  \{[m;0,j]:1\le j\le\ell-1\},&m\text{ odd},\\
  \{[m;i,\ell-i]:1\le i\le\ell-1\},&m\text{ even}.
 \end{cases}
\end{equation*}
Set
\[ \mathscr F_{\ell,m} :=\{\partial_1,\ldots,\partial_m\}\sqcup\mathscr B_{\ell,m}, \qquad \mathscr M_{\ell,m}:=\mathscr I_{\ell,m}\setminus\mathscr F_{\ell,m}. \]
Then
\[ |\mathscr F_{\ell,m}|=m+\ell-1, \qquad |\mathscr M_{\ell,m}| =\frac{(m-1)(\ell-1)(\ell+2)}2. \]

Figure~\ref{fig:seed-43} shows the quiver of $\Sigma_{4,3}$, including its
frozen vertices.

\begin{figure}[t]
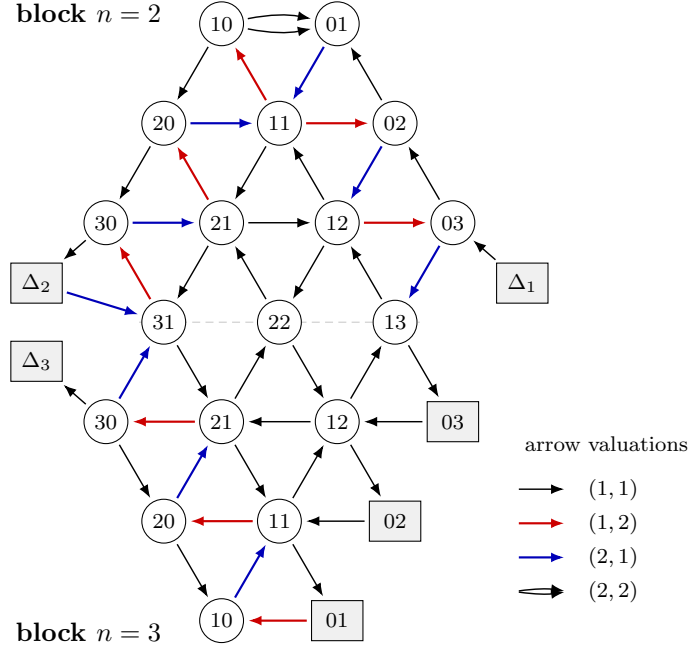

\centering
\Symfourthree
\caption{The folded seed $\Sigma_{4,3}$.  The upper and lower
triangles have indices $n=2$ and $n=3$; the vertices $31,22,13$ form
their common boundary.  A label $ij$ in triangle $n$ denotes
$z_{i,j}^{(n)}$, while a label on the common boundary denotes
$z_{i,4-i}^{(2)}=z_{i,4-i}^{(3)}$.  Circles are mutable and squares are
frozen.  The ordered arrow valuation is read from source to target; for a
frozen endpoint, the valuation is determined by the orbit-size symmetrizer.  Red arrows have valuation $(1,2)$ and blue arrows have valuation $(2,1)$;
the two parallel black arrows at the top have valuation $(2,2)$.}
\label{fig:seed-43}
\end{figure}

For odd $n\ge3$, triangles $n$ and $n-1$ are glued along the side $i+j=\ell$;
for even $n\ge4$, they are glued along the side $i=0$. A chamber fixed by
transposition gives an orbit of size one; all other chamber labels occur in
transpose pairs.  The remaining boundary of triangle $m$, together with
$\Delta_1,\ldots,\Delta_m$, is frozen.

Define
\[ B_{\ell,m}=(b_{uv})_{ \substack{u\in\mathscr M_{\ell,m}\\v\in\mathscr I_{\ell,m}}} \]
by the orbit-sum formula \eqref{eq:orbit-sum}.

\begin{proposition}
\label{prop:folded-B}
The matrix $B_{\ell,m}$ is well defined and has full row rank.  Its mutable
principal part is skew-symmetrizable.  If $d_u$ is the size
of the transpose orbit above $u$, then
\begin{equation}
 d_ub_{uv}=-d_vb_{vu}
 \qquad(u,v\in\mathscr M_{\ell,m}).
 \label{eq:symmetrizer}
\end{equation}
\end{proposition}

\begin{proof}
Equivariance shows that the matrix is well defined.  For mutable $u,v$,
\[ d_ub_{uv} =\sum_{u'\in\bar u}\sum_{v'\in\bar v} \widetilde b_{u'v'} =-d_vb_{vu}, \]
which proves \eqref{eq:symmetrizer}.

View $\widetilde B_{\ell,m}$ as a surjective, $\tau$-equivariant map
$\mathbb Q^{\widetilde I}\to\mathbb Q^{\widetilde M}$, where
$\widetilde I$ and $\widetilde M$ index all vertices and the mutable
vertices, respectively.  Taking invariants under a group of order two is exact in characteristic zero.  In
orbit-sum bases, the induced map is represented by $B_{\ell,m}$ up to
invertible diagonal rescaling by orbit sizes.  Hence the map represented by $B_{\ell,m}$ is surjective, so
$B_{\ell,m}$ has full row rank.
\end{proof}

\subsection{The folded exchange relations}
\label{subsec:folded-relations}

We use the conventions
\[ z_{0,0}^{(n)}:=1, \qquad z_{\ell,0}^{(n)}:=\Delta_n, \qquad z_{0,\ell}^{(2)}:=\Delta_1. \]
All variables on identified boundaries are interpreted through the identifications in
\eqref{eq:V-index}.

\begin{theorem}[folded exchange relations]
\label{thm:folded-relations}
For every mutable index $u$, the variable $z_u'$ obtained by mutation
is a homogeneous regular element of $\cR_{\ell,m}$.  The relations are:

\smallskip
\noindent\emph{(a) Variables on the base edge.}  For $2\le n\le m$ and
$1\le i\le\ell-2$,
\begin{equation}
 z_{i,0}^{(n)}(z_{i,0}^{(n)})'
 =z_{i+1,0}^{(n)}(z_{i-1,1}^{(n)})^2
  +z_{i-1,0}^{(n)}(z_{i,1}^{(n)})^2.
 \label{eq:base-edge}
\end{equation}
At the corner $i=\ell-1$,
\begin{align}
 z_{\ell-1,0}^{(n)}(z_{\ell-1,0}^{(n)})'
 &=\Delta_n(z_{\ell-2,1}^{(n)})^2
   +z_{\ell-2,0}^{(n)}(z_{\ell-1,1}^{(n)})^2,
 &&n\text{ even},
 \label{eq:base-edge-even-corner}\\
 z_{\ell-1,0}^{(n)}(z_{\ell-1,0}^{(n)})'
 &=(z_{\ell-2,1}^{(n)})^2
   +\Delta_nz_{\ell-2,0}^{(n)}(z_{\ell-1,1}^{(n)})^2,
 &&n\text{ odd}.
 \label{eq:base-edge-odd-corner}
\end{align}

\smallskip
\noindent\emph{(b) The other boundary of triangle $n=2$.}  For
$1\le j\le\ell-1$,
\begin{equation}
 z_{0,j}^{(2)}(z_{0,j}^{(2)})'
 =z_{0,j-1}^{(2)}(z_{1,j}^{(2)})^2
  +z_{0,j+1}^{(2)}(z_{1,j-1}^{(2)})^2.
 \label{eq:first-boundary}
\end{equation}

\smallskip
\noindent\emph{(c) Interior vertices.}  For $2\le n\le m$, $i,j>0$, and
$i+j<\ell$,
\begin{align}
 z_{i,j}^{(n)}(z_{i,j}^{(n)})'
 & =z_{i,j-1}^{(n)}z_{i+1,j}^{(n)}z_{i-1,j+1}^{(n)}
+z_{i,j+1}^{(n)}z_{i-1,j}^{(n)}z_{i+1,j-1}^{(n)}.
 \label{eq:interior}
\end{align}

\smallskip
\noindent\emph{(d) Identified boundaries with even $n$.}  Let $2\le n<m$ be even and let
$i,j\ge1$ with $i+j=\ell$.  Then
\begin{align}
 z_{i,j}^{(n)}(z_{i,j}^{(n)})'
 &=c_{n;i,j}
   z_{i-1,j}^{(n)}z_{i,j-1}^{(n+1)}
+z_{i,j-1}^{(n)}z_{i-1,j}^{(n+1)},
 \label{eq:even-gluing}
\end{align}
where
\begin{equation}
 c_{n;i,j}:=
 \begin{cases}
  \Delta_n,&(i,j)=(\ell-1,1),\\
  1,&\text{otherwise}.
 \end{cases}
 \label{eq:even-corner-factor}
\end{equation}

\smallskip
\noindent\emph{(e) Identified boundaries with odd $n$.}  Let $3\le n<m$ be odd and
$1\le j\le\ell-1$.  Then
\begin{align}
 z_{0,j}^{(n)}(z_{0,j}^{(n)})'
 &=\kappa_{n;j}z_{1,j}^{(n)}z_{1,j-1}^{(n+1)}
+z_{1,j-1}^{(n)}z_{1,j}^{(n+1)},
 \label{eq:odd-gluing}
\end{align}
where
\begin{equation}
 \kappa_{n;j}:=
 \begin{cases}
  \Delta_n,&j=\ell-1,\\
  1,&j<\ell-1.
 \end{cases}
 \label{eq:kappa}
\end{equation}
\end{theorem}

\begin{proof}
Apply Lemma~\ref{lem:restriction-folding} to
$\widetilde\Sigma_{\ell,m}$ and $\rho$.  Proposition
\ref{prop:transpose-comparison} identifies the two members of every
transpose pair.  Products of left and right chamber functions on common
bases become squares, giving \eqref{eq:base-edge}--
\eqref{eq:base-edge-odd-corner} and \eqref{eq:first-boundary}.  Interior relations restrict to
\eqref{eq:interior}; the relations along the shared sides give
\eqref{eq:even-gluing} and \eqref{eq:odd-gluing}.

The minimal lift in \cite{FeiKroneckerII} introduces a determinant
factor in three families of corner relations.
At the corner of the common base edge it gives the parity-dependent relations
\eqref{eq:base-edge-even-corner} and
\eqref{eq:base-edge-odd-corner}.  On the other boundary of triangle $n=2$, the correction is
$z_{0,\ell}^{(2)}=\Delta_1$.  The corrections at the two kinds of
identified boundary are \eqref{eq:even-corner-factor} and \eqref{eq:kappa}.  The unfolded mutated
variables are regular in $\cT_{\ell,m}$ by
Theorem~\ref{thm:flagged-matrix-form}; their restrictions are therefore regular
in $\cR_{\ell,m}$.  Homogeneity follows by restricting the unfolded grading.
\end{proof}

The five families account for
\begin{align*}
 &(m-1)(\ell-1)+(\ell-1)
 +(m-1)\frac{(\ell-1)(\ell-2)}2\\
 &\hspace{30mm}+(m-2)(\ell-1)
 =\frac{(m-1)(\ell-1)(\ell+2)}2
\end{align*}
mutable indices.  Homogeneity also gives
\[ B_{\ell,m}\operatorname{Wt}(\mathbf x_{\ell,m})^T=0, \]
where $\operatorname{Wt}(\mathbf x_{\ell,m})$ is the weight matrix of the
folded family.

\begin{corollary}[the folded seed]
\label{cor:folded-seed}
The pair
$\Sigma_{\ell,m}:=(B_{\ell,m},\mathbf x_{\ell,m})$
is a seed whose exchange matrix has full row rank and a skew-symmetrizable
principal part.  Each variable obtained by one mutation is a regular
homogeneous invariant.
\end{corollary}

\begin{proof}
Algebraic independence is Theorem~\ref{thm:folded-family-transcendence-basis};
Theorem~\ref{thm:folded-relations} gives the exchange relations, and
Proposition~\ref{prop:folded-B} gives full row rank and
skew-symmetrizability.
\end{proof}

The localization at $H_1$ inverts $\Delta_1=h_\ell$ but no other central
determinant.  The proof of the main theorem will use the analogous
localization for $C_2$ and then intersect the two localized rings.

\section{The upper cluster algebra}
\label{sec:upper-cluster-equality}

Let
\[ \cU_{\ell,m}:=\mathcal U(\Sigma_{\ell,m}) \]
be the upper cluster algebra of the folded seed, with frozen variables as
polynomial coefficients.  We prove $\cU_{\ell,m}\subseteq\cR_{\ell,m}$ from the initial seed and
its adjacent seeds, then obtain the reverse inclusion
after Gauss localization and remove the localization by intersection.

\subsection{The initial upper bound}
\label{subsec:upper-bound-regularity}

Let
$\mathscr I=\mathscr M_{\ell,m}\sqcup\mathscr F_{\ell,m}$
be the index set of the folded seed.  Its initial Laurent ring is
\[ \mathcal L_0 :=\kk[z_f:f\in\mathscr F_{\ell,m}] [z_u^{\pm1}:u\in\mathscr M_{\ell,m}]. \]
For a mutable index $u$, let $z_u'$ be its one-step mutation and put
\[ \mathcal L_u :=\kk[z_f:f\in\mathscr F_{\ell,m}] [(z_u')^{\pm1},z_v^{\pm1}:v\in\mathscr M_{\ell,m}\setminus\{u\}]. \]
The corresponding upper bound is
\[ \overline{\cU}_{\ell,m}:=\mathcal L_0\cap\bigcap_{u\in\mathscr M_{\ell,m}}\mathcal L_u\subseteq\Frac(\cR_{\ell,m}). \]

By Proposition~\ref{prop:folded-B}, the extended exchange matrix has full
row rank.  The ordinary-cluster specialization of
\cite[Theorem~3.11]{GSV}, over the polynomial ring in the frozen variables,
therefore gives
\begin{equation}
 \cU_{\ell,m}=\overline{\cU}_{\ell,m}.
 \label{eq:upper-bound-equality}
\end{equation}

\subsection{Initial irreducibility}
\label{subsec:initial-irreducibility}

For a polynomial dominant weight
$\lambda=(\lambda_1,\ldots,\lambda_\ell)$, the first row of its Young
diagram has length $\lambda_1$.  We use
the classical decomposition
\begin{equation}
 \Sym^d(\Sym^2V)
 \cong\bigoplus_{\rho\vdash d}\mathbf S_{2\rho}V,
 \label{eq:sym-sym2-decomposition}
\end{equation}
with summands of height greater than $\ell$ omitted; see
\cite[Chapter~I, Section~5, Example~4]{Macdonald}.

\begin{lemma}\label{lem:first-row-bound}
Let $0\ne f\in\cR_{\ell,m}$ be homogeneous for $T_V\times T_W$.
Its irreducible factors are homogeneous.  If $f$ is nonconstant and has
$T_V$-weight $\lambda$, then $\lambda_1\ge2$.  Consequently, a nonconstant $f$ is irreducible
whenever $\lambda_1\le3$.
\end{lemma}

\begin{proof}
Since $\cR_{\ell,m}$ is a UFD with constant units, the connected torus
$T_V\times T_W$ fixes each irreducible factor up to scalar; the factors are
therefore homogeneous.  The possible $T_V$-weights in multidegree $\alpha$
are the highest weights occurring in
$\bigotimes_{r=1}^m\Sym^{\alpha_r}(\Sym^2V). $
By \eqref{eq:sym-sym2-decomposition}, every nontrivial factor contributes a
partition whose first row has length at least two, and the same is true for
every Littlewood--Richardson constituent.  Thus the weight of a nonconstant homogeneous element has first row of
length at least two.  The weight of a product of two such elements has
first row of length at least four.
\end{proof}

\begin{corollary}
\label{cor:initial-irreducibility}
The initial functions below are irreducible in $\cR_{\ell,m}$:
\begin{enumerate}
 \item every central determinant $\Delta_r$;
 \item every base-edge function $z_{i,0}^{(n)}$;
 \item every chamber function with $s(v)=2$.
\end{enumerate}
\end{corollary}

\begin{proof}
The corresponding weights are
$2\omega_\ell$, $2\omega_i$, and $\omega_i+\omega_j+\omega_{i+j},$
whose first rows have length at most three.
\end{proof}

The remaining cases reduce to determinants of symmetric pairings.

\begin{lemma}\label{lem:symmetric-pairing-determinant}
Let $E$ be a vector space over a field of characteristic different from
two, and let $A,B\subseteq E$ be subspaces of the same positive
dimension $r$.  For $C\in\Sym^2(E^\vee)$, let
\[ p_{A,B}(C):=\det\bigl(C(b_i,a_j)\bigr)_{1\le i,j\le r}, \]
where bases of $A$ and $B$ have been chosen.  Then
$p_{A,B}$ is absolutely irreducible, up to a nonzero scalar depending
on the bases.
\end{lemma}

\begin{proof}
Let $N=\dim\Sym^2(E^\vee)$ and consider
\[ \mathcal I_{A,B} =\{([a],C)\in\mathbb P(A)\times\Sym^2(E^\vee):C(a,B)=0\}. \]
For $a\ne0$, the map $C\mapsto C(a,-)|_B$ from $\Sym^2(E^\vee)$ to
$B^\vee$ is surjective.  Thus $\mathcal I_{A,B}$ is a vector bundle over
$\mathbb P(A)$ of dimension $N-1$, hence irreducible.  Its proper image in
$\Sym^2(E^\vee)$ is the hypersurface $V(p_{A,B})$.

The polynomial $p_{A,B}$ is nonzero: choose bases adapted to
$A\cap B$ and a symmetric form whose pairing between $A$ and $B$ is
nonsingular.  The same bases admit a one-parameter family with pairing
matrix $\operatorname{diag}(1,\ldots,1,t)$.  Hence $V(p_{A,B})$ contains a
smooth corank-one point and is reduced.  It is therefore an irreducible
hypersurface.  The argument is unchanged after field extension, so
$p_{A,B}$ is absolutely irreducible.
\end{proof}

To apply Gauss's lemma, we first need two rank estimates.  Fix $s\ge3$
and work on the open set where
$C_1,\ldots,C_{s-1}$ are invertible.  Write
\begin{equation*}
G_s:=
 \begin{cases}
  C_{\overleftarrow{\mathbf q_s}},&s\text{ even},\\
  C_{s-1}^{-1}C_{s-2}\cdots C_2^{-1}C_1,
    &s\text{ odd},
 \end{cases}
\end{equation*}
where the products alternate as in \eqref{eq:alternating-word}; for odd
$s$, one has $C_{\overleftarrow{\mathbf q_s}}=C_sG_s$.

\begin{lemma}\label{lem:rank-drop-codimension}
Let $k=i+j$.
\begin{enumerate}
 \item If $s$ is even and $i,j>0$, then
 $\operatorname{codim} \{\operatorname{rank}(P_kG_sJ_j)<j\}\ge i+1.$
 \item If $s$ is odd, $j>0$, and $k<\ell$, then
 $\operatorname{codim} \{\operatorname{rank}[J_i\mid G_sJ_j]<k\} \ge\ell-k+1.$
\end{enumerate}
The codimensions are taken in the open set where all preceding
symmetric matrices are invertible.
\end{lemma}

\begin{proof}
Fix the preceding matrices other than $C_1$, and write $G_s=HC_1$
with $H$ invertible.  For every nonzero $x\in\kk^\ell$,
the evaluation map $C\mapsto Cx$ from $\SymMat_\ell$ to $\kk^\ell$ is
surjective.  In the even case, for each $[a]\in\mathbb P^{j-1}$ the
condition
$P_kHC_1J_ja=0$
imposes $k$ independent linear equations.  The incidence variety therefore
has codimension at least $k-(j-1)=i+1$.  In the odd case the condition
$HC_1J_ja\in J_i\kk^i$ imposes $\ell-i$ equations, giving codimension
$(\ell-i)-(j-1)=\ell-k+1$.
\end{proof}

\begin{lemma}
\label{lem:initial-functions-irreducible}
Every member of the folded initial family $\mathbf x_{\ell,m}$ is
irreducible in $\cR_{\ell,m}$.
\end{lemma}

\begin{proof}
It remains to consider chamber functions with $s=s(v)\ge3$.  Such a function has
$j>0$; if $s$ is even then $i>0$, while if $s$ is odd then $i+j<\ell$.
Let $A_{<s}$ be the polynomial ring in the entries of
$C_1,\ldots,C_{s-1}$ and $K_{<s}=\Frac(A_{<s})$.

Over $K_{<s}$, the factor $\Pi_s$ is a unit.  If $s$ is even, the matrix
$P_kG_sJ_j$ has generic rank $j$, and taking the quotient by its column span
converts the chamber determinant into the determinant of $C_s$ on two
$i$-dimensional subspaces.  If $s$ is odd, the matrix $[J_i\mid G_sJ_j]$ has
generic rank $k=i+j$, and the chamber function is the determinant of $C_s$ on
two $k$-dimensional subspaces.  In either case,
Lemma~\ref{lem:symmetric-pairing-determinant} gives irreducibility in
$K_{<s}[\SymMat_\ell]$.

An irreducible common factor of the coefficients would define a divisor
in the locus where the rank drops on $\prod_{r<s}\Delta_r\ne0$, contradicting
Lemma~\ref{lem:rank-drop-codimension}.  It must therefore be associated to
some $\Delta_r$, $r<s$.  But \eqref{eq:W-degree} gives degree $j$ or
$\ell-j$ in $C_r$, both smaller than $\deg_{C_r}\Delta_r=\ell$.  Gauss's
lemma proves irreducibility after adjoining $C_s$; adjoining the later
matrix entries preserves primality.  Hence the chamber function is
irreducible in $\kk[X_{\ell,m}]$ and therefore in $\cR_{\ell,m}$.
\end{proof}

\subsection{Coprimality with one-step mutations}
\label{subsec:first-mutation-coprimality}

For a homogeneous element $f$, write
$\deg(f)=(\deg_W(f),\wt_V(f))$.  If
$z_uz_u'=M_{u,+}+M_{u,-} $
is one of the exchange relations of Theorem~\ref{thm:folded-relations},
put
\[ \Gamma_u:=\deg(M_{u,+})-2\deg(z_u). \]
Homogeneity makes the definition independent of the chosen exchange monomial.

\begin{lemma}\label{lem:quotient-degree-obstruction}
If $z_u\mid z_u'$ in $\cR_{\ell,m}$, then the homogeneous component of
$\cR_{\ell,m}$ of bidegree $\Gamma_u$ is nonzero.  In particular, the
$T_W$-part of $\Gamma_u$ is nonnegative and its $T_V$-part is a
polynomial dominant weight.
\end{lemma}

\begin{proof}
The quotient $z_u'/z_u$ would be a nonzero homogeneous element of degree
$\Gamma_u$.
\end{proof}

Substitution of \eqref{eq:V-weight}--\eqref{eq:W-degree} gives
\begin{equation*}
\begin{array}{c|c|c}
  \text{variable}&\text{remaining case}&\text{obstruction otherwise}\\
  \hline
  z_{i,0}^{(n)}&i=1&[\omega_i]\Gamma=-2\ (i\ge2)\\
  z_{0,j}^{(2)}&j=1&[\omega_j]\Gamma=-2\ (j\ge2)\\
  z_{i,j}^{(n)}\text{ interior}&(i,j)=(1,2),(2,1)
    &\text{$T_V$-weight not dominant}\\
  \text{variable on an identified boundary}&\text{none}&[\delta_1]\Gamma=-1
 \end{array}
\end{equation*}
where $[\omega_a]\Gamma$ and $[\delta_r]\Gamma$ denote coefficients in
the indicated bases.

\begin{lemma}\label{lem:degree-reduction}
The quotient-degree obstruction proves $z_u\nmid z_u'$ in every mutable
direction except possibly:
\begin{equation}
 z_{1,0}^{(n)},\qquad z_{0,1}^{(2)},\qquad
 z_{1,2}^{(n)},\qquad z_{2,1}^{(2)}.
 \label{eq:degree-exception-list}
\end{equation}
Here the last two families occur only when their indices are defined.
The two $n=2$ interior exceptions
$z_{1,2}^{(2)}$ and $z_{2,1}^{(2)}$ are also coprime to their respective mutated variables.
\end{lemma}

\begin{proof}
For a mutable vertex $u=[n;i,0]$ on the base edge with $i\ge2$, the coefficient of $\omega_i$ in
the $T_V$-part of $\Gamma_u$ is $-2$.  The same argument on the other
$n=2$ boundary gives coefficient $-2$ at $\omega_j$ for
$j\ge2$.  Only $i=1$ and $j=1$, respectively, remain.

For an interior mutable vertex $u=[n;i,j]$, put
\[ F(a):=\omega_{a-1}+\omega_{a+1}-\omega_a, \qquad \omega_0:=0. \]
A direct substitution in the weight formula gives
\[ (\Gamma_u)_V=F(i)+F(j)+F(i+j)+c_u\omega_\ell, \]
where $c_u\in\mathbb Z$ accounts for the identification of a
neighboring boundary variable with one from the preceding triangle.  Since an interior mutable index satisfies
$i+j<\ell$, this correction changes only the coefficient of
$\omega_\ell$ and cannot alter any of the negative coefficients below
${\omega_\ell}$.  The lower coefficients are nonnegative only for $\,(i,j)=(1,2)$ or $(2,1)$.  Assume $i\le j$.  If $i\ge3$, the coefficient at $\omega_i$ is negative
unless $j=i+1$, and in that remaining case the coefficient at
$\omega_{i+j}$ is negative.  If $i=2$, the coefficient at
$\omega_{i+j}$ is negative.  If $i=1$, the coefficient at
$\omega_1$ is nonnegative only when $j=2$.  Interchanging $i$ and
$j$ gives the stated pair.  A direct substitution of the
$T_W$-degrees shows that the $(2,1)$-case has a negative coordinate for
every $n>2$.  For every mutable vertex on an identified boundary, for either parity, the coefficient of
$\delta_1$ in $\Gamma_u$ is $-1$.  This proves the reduction to
\eqref{eq:degree-exception-list}.

For $(i,j)=(1,2)$ or $(2,1)$ in triangle $n=2$, the quotient would have
$T_V$-weight
$\omega_2+\omega_4=(2,2,1,1)$
and $T_W$-multidegree $2\delta_1+\delta_2$ or
$\delta_1+2\delta_2$.  Its multiplicity is therefore that of
$\mathbf S_{(2,2,1,1)}V$ in
$\Sym^2(\Sym^2V)\otimes\Sym^2V.$
Since
\[ \Sym^2(\Sym^2V)=\mathbf S_{(4)}V\oplus\mathbf S_{(2,2)}V, \]
the Pieri rule \cite[Chapter~I]{Macdonald} shows that this multiplicity is zero: the partition
$(2,2,1,1)$ is not obtained from either $(4)$ or $(2,2)$ by adding a
horizontal two-strip.  Lemma~\ref{lem:quotient-degree-obstruction} settles
these two cases.
\end{proof}

The degree argument leaves only the three families
\begin{equation}
 z_{1,0}^{(n)},\qquad z_{0,1}^{(2)},\qquad z_{1,2}^{(n)}\quad(n\ge3),
 \label{eq:exceptional-coprimality-families}
\end{equation}
whenever the indicated indices are mutable.

\begin{lemma}\label{lem:transverse-curve-criterion}
Suppose that $z_uz_u'=N_u$ and that a curve
$\gamma:\mathbb A^1\to X_{\ell,m}$ satisfies
\[ \operatorname{ord}_0(z_u\circ\gamma) =\operatorname{ord}_0(N_u\circ\gamma)=1. \]
Then $z_u\nmid z_u'$ in $\cR_{\ell,m}$.
\end{lemma}

\begin{proof}
Write $z_u\circ\gamma=t\eta(t)$ and
$N_u\circ\gamma=tq(t)$ with $\eta(0)q(0)\ne0$.  Then
$z_u'\circ\gamma=q/\eta$ is nonzero at the origin, whereas
$z_u\circ\gamma$ vanishes there.
\end{proof}

\begin{lemma}\label{lem:exceptional-directions}
For every mutable variable in \eqref{eq:exceptional-coprimality-families}, we have
\[ z_u\nmid z_u'. \]
\end{lemma}

\begin{proof}
Appendix~\ref{app:exceptional-coprimality} gives a transverse curve for
each family.  Apply Lemma~\ref{lem:transverse-curve-criterion}.
\end{proof}

\begin{lemma}\label{lem:initial-coprimality}
The initial family consists of pairwise nonassociate prime elements of
$\cR_{\ell,m}$, and for every mutable index $u$,
\begin{equation}
 \gcd(z_u,z_u')=1.
 \label{eq:initial-coprime}
\end{equation}
Consequently,
\begin{equation}
 \cU_{\ell,m}\subseteq\cR_{\ell,m}.
 \label{eq:upper-bound-inclusion}
\end{equation}
\end{lemma}

\begin{proof}
Lemma~\ref{lem:initial-functions-irreducible} and factoriality make every
initial function prime, and algebraic independence makes distinct initial
functions nonassociate.  Lemmas~\ref{lem:degree-reduction} and
\ref{lem:exceptional-directions} prove
\eqref{eq:initial-coprime}.

Let $\nu_u$ be the valuation along $V(z_u)$.  All generators inverted in
$\mathcal L_u$ have $\nu_u$-value zero, so every element of
$\mathcal L_u$ has nonnegative $\nu_u$-valuation.  An element of the upper
bound lies in $\mathcal L_0$, so its poles can occur only along the
divisors of mutable initial variables,
and membership in the corresponding adjacent Laurent rings excludes
these poles.  Frozen variables are not inverted.  Thus every height-one valuation of $\cR_{\ell,m}$ is nonnegative on
the upper bound.  Since a UFD is the intersection
of its height-one valuation rings, the upper bound lies in
$\cR_{\ell,m}$.  Together with \eqref{eq:upper-bound-equality}, this proves the claim.
\end{proof}

\subsection{Restriction and Gauss localizations}
\label{subsec:restriction-upper-bound}

For a mutable folded index $u$, choose a representative in the unfolded seed
$\widetilde u$.  If its transpose orbit has two members, define the
\emph{orbit mutation}
\[ \mu_{\bar u}:=\mu_{\widetilde u}\mu_{\tau\widetilde u}; \]
if the orbit has one member, use the single mutation $\mu_{\widetilde u}$.  The
two mutations in a nontrivial orbit commute because the two orbit members
are not adjacent in the unfolded quiver.

\begin{proposition}
\label{prop:source-restriction-in-U}
Restriction to symmetric matrices satisfies
\[ \rho(\cT_{\ell,m})\subseteq\cU_{\ell,m}. \]
\end{proposition}

\begin{proof}
Every $F\in\cT_{\ell,m}=\mathcal U(\widetilde\Sigma_{\ell,m})$ is Laurent
in the unfolded initial seed and in each orbit-mutated seed.  Symmetric
restriction sends the initial Laurent ring to $\mathcal L_0$, and
Lemma~\ref{lem:restriction-folding} sends the orbit-mutated Laurent ring in
direction $u$ to $\mathcal L_u$.  Therefore
\[ \rho(F)\in\mathcal L_0\cap \bigcap_{u\in\mathscr M_{\ell,m}}\mathcal L_u =\cU_{\ell,m}. \]
\end{proof}

For $1\le s\le m$ and $0\le i\le\ell$, set
\[ h_i^{[s]}:=\det C_s[I_i,I_i], \qquad h_0^{[s]}:=1, \qquad H_s:=\prod_{i=1}^{\ell}h_i^{[s]}. \]
Thus $h_i^{[1]}=h_i$ and $h_\ell^{[s]}=\Delta_s$.  Define
$\widetilde h_i^{[s]}$ and $\widetilde H_s$ on $Y_{\ell,m}$ by the
same formulas.

\begin{proposition}
\label{prop:all-localized-restrictions}
For every $s$, restriction induces a surjection
\[ \rho_{H_s}:\cT_{\ell,m}[\widetilde H_s^{-1}] \twoheadrightarrow\cR_{\ell,m}[H_s^{-1}]. \]
\end{proposition}

\begin{proof}
Relabel $C_s$ as the first matrix and apply the two Gauss-chart
propositions.  On the quotient slices the map is the surjective restriction
\[ \kk[T_\ell\times\Mat_\ell^{m-1}] \longrightarrow \kk[T_\ell\times\SymMat_\ell^{m-1}]. \]
\end{proof}

\begin{proposition}
\label{prop:localized-equality}
For every $1\le s\le m$, we have
\[ \cR_{\ell,m}[H_s^{-1}] =\cU_{\ell,m}[H_s^{-1}]. \]
\end{proposition}

\begin{proof}
The functions $h_i^{[s]}$ belong to the initial seed, so $H_s\in\cU_{\ell,m}$.
The inclusion \eqref{eq:upper-bound-inclusion} gives one containment after
localization.  For the other, Proposition~\ref{prop:source-restriction-in-U}
gives
\[ \rho\bigl(\cT_{\ell,m}[\widetilde H_s^{-1}]\bigr) \subseteq\cU_{\ell,m}[H_s^{-1}], \]
and Proposition~\ref{prop:all-localized-restrictions} identifies the
left-hand side with $\cR_{\ell,m}[H_s^{-1}]$.
\end{proof}

\subsection{Removing the Gauss localization}
\label{subsec:remove-gaussian-localization}

We use the first two matrices.  In the first folded triangle,
\[ h_i^{[1]}=z_{0,i}^{(2)}, \qquad h_i^{[2]}=z_{i,0}^{(2)}. \]
For $1\le k\le\ell-1$, denote their exchange polynomials by
\begin{align*}
 E_k^{[1]}
 &:=h_{k-1}^{[1]}(z_{1,k}^{(2)})^2
   +h_{k+1}^{[1]}(z_{1,k-1}^{(2)})^2,\\
 E_k^{[2]}
 &:=h_{k+1}^{[2]}(z_{k-1,1}^{(2)})^2
   +h_{k-1}^{[2]}(z_{k,1}^{(2)})^2.
\end{align*}
Thus
\[ h_k^{[1]}(h_k^{[1]})'=E_k^{[1]}, \qquad h_k^{[2]}(h_k^{[2]})'=E_k^{[2]}. \]

\begin{lemma}
\label{lem:R-localization-intersection}
Inside $\Frac(\cR_{\ell,m})$, we have
\[ \cR_{\ell,m}[H_1^{-1}] \cap\cR_{\ell,m}[H_2^{-1}] =\cR_{\ell,m}. \]
\end{lemma}

\begin{proof}
The factors of $H_1$ and $H_2$ are irreducible by
Lemma~\ref{lem:initial-functions-irreducible}; they are pairwise
nonassociate because their $T_W$-multidegrees are positive multiples of
$\delta_1$ and $\delta_2$, respectively.  Thus $H_1$ and $H_2$
are coprime in the UFD $\cR_{\ell,m}$, and the assertion is the standard
coprime-localization identity.
\end{proof}

\begin{lemma}
\label{lem:U-localization-intersection}
One has
\[ \cU_{\ell,m}[H_1^{-1}] \cap\cU_{\ell,m}[H_2^{-1}] =\cU_{\ell,m}. \]
\end{lemma}

\begin{proof}
Equation~\eqref{eq:upper-bound-equality} expresses $\cU_{\ell,m}$ as a
finite intersection of the initial and adjacent Laurent rings $\mathcal L_v$
(with $v=0$ for the initial ring).  Since a finite intersection is the kernel
of the
difference map from the direct sum to the ambient field, flatness of
localization gives
\begin{equation}
 \cU_{\ell,m}[H_s^{-1}]
 =\bigcap_v\mathcal L_v[H_s^{-1}]
 \qquad(s=1,2).
 \label{eq:localized-upper-bound-star}
\end{equation}
It is enough to prove that $H_1$ and $H_2$ are coprime in every
$\mathcal L_v$.

In the initial ring, and after mutation away from the two boundary
families, $H_1$ and $H_2$ are associated to the polynomial variables
$\Delta_1$ and $\Delta_2$.  After mutation at $h_k^{[1]}$,
\[ H_1\sim\Delta_1E_k^{[1]},\qquad H_2\sim\Delta_2. \]
The polynomial $E_k^{[1]}$ is a sum of two distinct Laurent monomials and
is not divisible by $\Delta_2$.  The analogous mutation at $h_k^{[2]}$
gives
\[ H_1\sim\Delta_1,\qquad H_2\sim\Delta_2E_k^{[2]}, \]
with $E_k^{[2]}$ not divisible by $\Delta_1$.  Each $\mathcal L_v$ is a
UFD, so the coprime-localization identity applies in every Laurent ring in this
intersection.
Intersecting by \eqref{eq:localized-upper-bound-star} proves the claim.
\end{proof}

\begin{theorem}[main theorem]
\label{thm:main-equality}
For all $\ell,m\ge2$, the folded seed $\Sigma_{\ell,m}$ defines the
highest-weight invariant ring as an upper cluster algebra:
\[ \cR_{\ell,m}=\mathcal U(\Sigma_{\ell,m}). \]
\end{theorem}

\begin{proof}
By Proposition~\ref{prop:localized-equality}, the localizations at $H_1$ and
$H_2$ agree.  Lemmas~\ref{lem:R-localization-intersection} and
\ref{lem:U-localization-intersection} therefore give
\[ \cR_{\ell,m} =\cR_{\ell,m}[H_1^{-1}]\cap\cR_{\ell,m}[H_2^{-1}] =\cU_{\ell,m}[H_1^{-1}]\cap\cU_{\ell,m}[H_2^{-1}] =\cU_{\ell,m}. \]
\end{proof}

\section{Reddening sequences}
\label{sec:reddening}

Let $B_{\ell,m}^{\mathrm{uf}} :=(b_{uv})_{u,v\in\mathscr M_{\ell,m}} $ be the
mutable principal part of the folded exchange matrix.  A reddening sequence
for $\Sigma_{\ell,m}$ is a green-to-red sequence for the principal-coefficient
framing of $B_{\ell,m}^{\mathrm{uf}}$; frozen coefficients do not enter the
argument.

We use the class $\mathscr P'$ of skew-symmetrizable matrices generated from
the $1\times1$ zero matrix by mutation and source--sink extension.
Equivalently, a matrix lies in $\mathscr P'$ if and only if its quiver can be
reduced to one vertex by mutations and deletions of sources or sinks.
String-diagram exchange matrices lie in $\mathscr P'$, and every matrix in
$\mathscr P'$ admits a reddening sequence \cite[Theorem~4.8 and
Corollary~4.9]{CaoString}.

To reduce the folded matrix, we remove each transpose-fixed diamond diagonal
directly.  Between two such removals, the required disk flips occur in
disjoint transpose-exchanged regions and descend by orbit mutation.

\subsection{Orbit mutation on nonadjacent sets}

Let $\widetilde B$ be a mutable principal matrix carrying an admissible
involution $\tau$, in the sense of
\eqref{eq:admissible-equivariance}--\eqref{eq:admissible-sign}, and let
$B=\operatorname{Fold}_{\tau}(\widetilde B)$ be its orbit-sum fold.

\begin{lemma}\label{lem:orbit-mutation-folding}
For a mutable orbit $\bar k$, the mutations at the members of $\bar k$
commute and
\[ \operatorname{Fold}_{\tau} \left(\prod_{k'\in\bar k}\mu_{k'}(\widetilde B)\right) =\mu_{\bar k}(B). \]
If the members of $\bar k$ are all sources, or all sinks, then $\bar k$ is
a source, respectively a sink, of the folded quiver, and deletion
commutes with folding.
\end{lemma}

\begin{proof}
The members of the orbit are not adjacent, so their mutations commute.  For
$\bar i,\bar j\ne\bar k$, sum the simultaneous mutation formula over the
column orbit $\bar j$ and the mutation orbit $\bar k$.  Condition~\eqref{eq:admissible-sign} allows positive and negative
parts to pass through both sums and gives
\[ b'_{\bar i\bar j} =b_{\bar i\bar j} +[b_{\bar i\bar k}]_+[b_{\bar k\bar j}]_+ -[-b_{\bar i\bar k}]_+[-b_{\bar k\bar j}]_+. \]
Entries whose row or column index is $\bar k$ change sign on both sides.  The source--sink
statement follows by summing entries of one common sign.
\end{proof}

\begin{lemma}\label{lem:separated-pair}
Suppose the remaining vertices of a $\tau$-equivariant matrix decompose as
\[ \widetilde I=L\sqcup F\sqcup R, \qquad \tau(L)=R, \qquad \tau|_F=\mathrm{id}, \qquad \widetilde B_{L,R}=0. \]
If every vertex in $\mathbf w=(k_1,\ldots,k_s)$ belongs to $L$, then the paired sequence
$(k_1,\tau k_1),\ldots,(k_s,\tau k_s)$
remains admissible and folds to $(\bar k_1,\ldots,\bar k_s)$.
\end{lemma}

\begin{proof}
A mutation in $L$ cannot create an arrow between $L$ and $R$, because
the mutated vertex has no neighbor in $R$; the same argument applies in $R$.
Thus each pair of mutations preserves the absence of arrows between
$L$ and $R$, as well as equivariance.  For each two-element orbit, a vertex in $L$ or $R$ can be adjacent only
to its member on the same side.  For $f\in F$ and $k\in L$, equivariance
gives $\widetilde b_{f,k}=\widetilde b_{f,\tau k}$.
The sign condition for admissibility is therefore preserved, and
Lemma~\ref{lem:orbit-mutation-folding} applies at every step.
\end{proof}

\subsection{Removing a folded half-diamond diagonal}

Put $N=\ell-1$ and write
\[ d_i^{(r)}=[r;i,0] \qquad(1\le i\le N) \]
for the transpose-fixed diagonal vertices in triangle $r$.  For even $r$
set
\begin{equation}
 \mathbf p_i^+(r)
 :=\bigl([r;i+u,q-u]\bigr)_{
 {\substack{q=N-i,N-i-1,\ldots,1\\u=0,1,\ldots,q}}},
 \label{eq:positive-removal-sequence}
\end{equation}
and for odd $r$ set
\[ \mathbf p_i^-(r) :=\bigl([r;i-q,j]\bigr)_{ {\substack{q=i-1,i-2,\ldots,1\\j=q,q-1,\ldots,0}}}. \]
The indices are read in the stated order; the corresponding mutations are
performed one anti-diagonal at a time.  Figure~\ref{fig:half-diamond-mutation-order}
shows the order, and Appendix~\ref{app:triangular-mutation} proves the row
formula used below.

\begin{lemma}[removal of a folded half-diamond diagonal]
\label{lem:folded-half-diamond-removal}
Assume that the diagonal labelled $r$ is exposed in the disk model of \cite{FeiKroneckerII}.

If $r$ is even, then, successively for $i=1,\ldots,N$, mutation by
$\mathbf p_i^+(r)$ makes $d_i^{(r)}$ a source or sink, with
\begin{equation}
 \operatorname{supp}b_{d_i^{(r)},\bullet}=
 \begin{cases}
  \{d_{i+1}^{(r)}\},&i<N,\\
  \{[r;N-1,1]\},&i=N,
 \end{cases}
 \qquad
 \begin{cases}
  |b_{d_i^{(r)},d_{i+1}^{(r)}}|=1,&i<N,\\
  |b_{d_N^{(r)},[r;N-1,1]}|=2.&
 \end{cases}
 \label{eq:positive-removal-row}
\end{equation}
Delete $d_i^{(r)}$ and apply the mutations in reverse order.  The matrix on
the surviving vertices is then the preceding matrix with $d_i^{(r)}$ removed.

If $r$ is odd, the same statement holds for $i=N,N-1,\ldots,1$ with
$\mathbf p_i^-(r)$ and
\begin{equation}
 \operatorname{supp}b_{d_i^{(r)},\bullet}=
 \begin{cases}
  \{d_{i-1}^{(r)}\},&i>1,\\
  \{[r;1,1]\},&i=1,
 \end{cases}
 \qquad
 \begin{cases}
  |b_{d_i^{(r)},d_{i-1}^{(r)}}|=1,&i>1,\\
  |b_{d_1^{(r)},[r;1,1]}|=2.&
 \end{cases}
 \label{eq:negative-removal-row}
\end{equation}
\end{lemma}

\begin{proof}
For even $r$, after deleting
$d_1^{(r)},\ldots,d_{i-1}^{(r)}$, the mutable rows are those computed in
Appendix~\ref{app:triangular-mutation}, up to simultaneous reversal of
all arrows.  Proposition~\ref{prop:triangular-mutation-sequence} gives
\eqref{eq:positive-removal-row}.  By the exchange relations within each triangle and along the identified
boundaries, every additional neighbor of a vertex in the mutation sequence lies on
the adjacent unmutated side.  The corresponding columns are retained in
Appendix~\ref{app:triangular-mutation}.  Thus every
entry coming from an identified boundary is included in the calculation.
After the source or sink is deleted, passage to the principal submatrix
on the surviving vertices commutes with every remaining mutation; the reverse sequence is therefore the inverse of the
forward sequence on the restricted matrix.

The odd half-diamond is obtained from the even one by
\[ (i,j)\longmapsto(\ell-i-j,j) \]
and global arrow reversal.  This sends the order $1,\ldots,N$ to
$N,\ldots,1$ and $\mathbf p_{N+1-i}^+(r)$ to $\mathbf p_i^-(r)$, proving
\eqref{eq:negative-removal-row}.
\end{proof}

\subsection{The diagonal reduction}

A flip of a disk diagonal in an $\ell$-triangulated quadrilateral is the
standard Fock--Goncharov mutation sequence \cite[Section~10.3 and
Proposition~10.1]{FockGoncharov}.  In \cite[Step~2 in the proof of
Theorem~6.22]{FeiKroneckerII}, this sequence is applied at the
transpose-paired diagonals $[k]$ and $\overleftarrow{[k]}$ before the next
diamond diagonal is removed.

\begin{lemma}[folding paired disk flips]
\label{lem:flip-support-separation}
After the exposed transpose-fixed diagonal has been deleted, the interiors
of the two quadrilaterals supporting the paired flip of \cite{FeiKroneckerII} have no arrows
between them.  Their standard flip sequences are transpose images and mutate no fixed
boundary vertex.  Hence the paired flip descends to the folded
matrix.
\end{lemma}

\begin{proof}
The Fock--Goncharov quiver is the sum of the oriented quivers of the small
triangles.  Once the common fixed diagonal is removed, no small triangle
contains an interior vertex from each quadrilateral.  Their interiors therefore have vertex sets $L$ and $R=\tau L$ with
no arrows between them; the shared boundary is fixed and
unmutated.  Apply Lemma~\ref{lem:separated-pair}.
\end{proof}

\begin{proposition}[folded diagonal reduction]
\label{prop:direct-folded-diagonal-reduction}
If $m\ge3$ is odd, then $B_{\ell,m}^{\mathrm{uf}}$ can be reduced by
mutations and source--sink deletions to the string-diagram exchange matrix
for the Cartan matrix $A_{\ell-1}$ associated with a triangulation of an
$(m+1)$-gon.
\end{proposition}

\begin{proof}
Remove the outer diagonal $m$ by
Lemma~\ref{lem:folded-half-diamond-removal}.  Inductively, suppose that the
matrix obtained at this point is the fold of the disk matrix in \cite{FeiKroneckerII}
after deleting diagonals
$m,m-1,\ldots,k+1$.  Apply the folded paired flip of
Lemma~\ref{lem:flip-support-separation}, remove the now-exposed diagonal $k$ by
Lemma~\ref{lem:folded-half-diamond-removal}, and undo the flip.  Reversing the preparatory sequences restores all entries between surviving
vertices, so the
induction hypothesis is preserved with diagonal $k$ deleted.

After all fixed diamond diagonals are removed, the disk decomposition in
\cite{FeiKroneckerII}
leaves two disconnected transpose-exchanged copies of the
$\ell$-triangulation of an $(m+1)$-gon
\cite[Step~3 in the proof of Theorem~6.22]{FeiKroneckerII}.  Folding leaves
one copy, whose mutable matrix is the stated string-diagram matrix.
\end{proof}

\begin{lemma}[compatibility as $m$ increases]
\label{lem:successive-principal-submatrix}
Let $\iota_m$ send every orbit label represented in a triangle
$n\le m$ to the same orbit label in $\mathscr V_{\ell,m+1}$.  Its restriction to
$\mathscr M_{\ell,m}$ is injective, and
\begin{equation}
 b^{(\ell,m+1)}_{\iota_m(u),\iota_m(v)}
 =b^{(\ell,m)}_{u,v}
 \qquad(u,v\in\mathscr M_{\ell,m}).
 \label{eq:successive-m-entry-equality}
\end{equation}
Consequently $B_{\ell,m}^{\mathrm{uf}}$ is the principal submatrix of
$B_{\ell,m+1}^{\mathrm{uf}}$ on $\iota_m(\mathscr M_{\ell,m})$.
\end{lemma}

\begin{proof}
The map is the identity on representatives from triangles $2,\ldots,m$.
The only new equivalence relation identifies the common boundary of
triangles $m$ and $m+1$, as follows.

If $m$ is even, then $m+1$ is odd and
\begin{equation}
 (m+1;i,\ell-i)\sim(m;i,\ell-i)
 \qquad(1\le i\le\ell-1).
 \label{eq:successive-even-old-boundary}
\end{equation}
The labels on the right form the frozen outer boundary of
$\Sigma_{\ell,m}$; after the new triangle is attached they form an internal
boundary.  If $m$ is odd, then $m+1$ is even and
\begin{equation}
 (m+1;0,j)\sim(m;0,j)
 \qquad(1\le j\le\ell-1),
 \label{eq:successive-odd-old-boundary}
\end{equation}
again turning the old frozen outer boundary into an internal boundary.  Thus no two labels that were mutable in $\Sigma_{\ell,m}$ are newly
identified, which proves injectivity on $\mathscr M_{\ell,m}$.

Every exchange relation involving only variables from triangles
$2,\ldots,m-1$ is unchanged.
In triangle $m$, the base-edge and interior relations
\eqref{eq:base-edge}--\eqref{eq:interior} are also unchanged.
The old frozen outer boundary becomes mutable when the new triangle is
attached.  Its exchange relations are \eqref{eq:even-gluing} when $m$ is even and
\eqref{eq:odd-gluing} when $m$ is odd.  In either formula, every new factor is a variable from the new
triangle or one of the old outer-boundary labels in
\eqref{eq:successive-even-old-boundary} or
\eqref{eq:successive-odd-old-boundary}.  Thus the exchange relations at old mutable vertices retain the same
exponents of all old mutable variables.  This proves \eqref{eq:successive-m-entry-equality}.
Deleting the former frozen boundary labels and the labels belonging
only to triangle $m+1$ therefore leaves $B_{\ell,m}^{\mathrm{uf}}$.
\end{proof}

\begin{theorem}\label{thm:reddening}
For every $\ell,m\ge2$, the seed $\Sigma_{\ell,m}$ admits a reddening
sequence.
\end{theorem}

\begin{proof}
Suppose first that $m$ is odd.  Proposition
\ref{prop:direct-folded-diagonal-reduction} reduces
$B_{\ell,m}^{\mathrm{uf}}$ to a string-diagram matrix.  The latter lies in
$\mathscr P'$; reading the reduction backwards shows that
$B_{\ell,m}^{\mathrm{uf}}\in\mathscr P'$.  It therefore has a reddening
sequence \cite[Theorem~4.8 and Corollary~4.9]{CaoString}.

If $m$ is even, then $m+1$ is odd and
$B_{\ell,m+1}^{\mathrm{uf}}$ has a reddening sequence.  By
Lemma~\ref{lem:successive-principal-submatrix},
$B_{\ell,m}^{\mathrm{uf}}$ is a principal submatrix of it.  Green-to-red
sequences pass to principal submatrices of skew-symmetrizable matrices
\cite[Lemma~4.9]{CaoLi}.
\end{proof}

\section{Theta bases and polyhedral formulas for symmetric-square plethysm}
\label{sec:polyhedral-plethysm}

We work over $\mathbb C$ in the theta-function arguments.
Lemma~\ref{lem:theta-characteristic-zero-descent} descends the resulting bases
to the ground field fixed in Section~\ref{sec:plethysm-highest-weight}.
Gross--Hacking--Keel--Kontsevich construct theta functions from cluster
scattering diagrams and prove a corrected form of the Fock--Goncharov
dual-basis conjecture \cite[Introduction, Sections~7--8]{GHKK}.
Cheung--Magee--Mandel--Muller prove valuative independence and theta
reciprocity for seed data and obtain theta bases for anticanonical partial
compactifications \cite{CMMM}.

\subsection{The folded seed datum and the open theta basis}

Invert the frozen variables and set
\[ \cU_{\ell,m}^{\circ} :=\cU_{\ell,m}[z_f^{-1}:f\in\mathscr F_{\ell,m}]. \]
Let $\mathcal A_{\ell,m}^{\circ}$ be the open cluster $A$-variety obtained
by inverting every frozen coordinate.  For each seed $t$ in the mutation
class, put
\[ \mathcal A_t :=\operatorname{Spec} \kk[z_{u;t}^{\pm1}:u\in\mathscr M_{\ell,m}] [z_f:f\in\mathscr F_{\ell,m}], \]
and glue these charts by the usual mutable cluster transformations.  The
resulting scheme is denoted by $\mathcal A_{\ell,m}$.

The initial chart identifies the theta lattice with
\[ M_{\ell,m}:=\mathbb Z^{\mathscr I_{\ell,m}}, \qquad M_{\ell,m,\mathbb R}:=M_{\ell,m}\otimes_{\mathbb Z}\mathbb R, \qquad N_{\ell,m}:=\operatorname{Hom}(M_{\ell,m},\mathbb Z). \]
Write $e_v$ for the standard basis of $M_{\ell,m}$ and $e_v^\vee$ for the
dual basis of $N_{\ell,m}$.  Let
\[ L_{\ell,m}:=\mathbb Z^{\mathscr M_{\ell,m}} \]
with basis $\epsilon_u$ indexed by mutable vertices.

\begin{proposition}[the folded seed datum]
\label{prop:theta-seed-datum}
Define homomorphisms
\begin{equation*}
 \begin{aligned}
 P:L_{\ell,m}&\longrightarrow M_{\ell,m},
 &P(\epsilon_u)&=\sum_{v\in\mathscr I_{\ell,m}}b_{uv}e_v,\\
 Q:L_{\ell,m}&\longrightarrow N_{\ell,m},
 &Q(\epsilon_u)&=e_u^\vee.
 \end{aligned}
\end{equation*}
Let $D=\operatorname{diag}(d_u)$, where $d_u$ is the transpose-orbit size
from Proposition~\ref{prop:folded-B}, and put
\begin{equation}
 \begin{aligned}
 N_{\ell,m}^{\bullet}
 &:=N_{\ell,m}
   +\sum_{u\in\mathscr M_{\ell,m}}
     \mathbb Z\frac{e_u^\vee}{d_u},\\
 M_{\ell,m}^{\bullet}
 &:=\bigoplus_{u\in\mathscr M_{\ell,m}}\mathbb Z\,d_ue_u
   \ \oplus\!
   \bigoplus_{f\in\mathscr F_{\ell,m}}\mathbb Z\,e_f,\\
 Q^{\bullet}&:=QD^{-1},
 &P^{\bullet}&:=PD.
 \end{aligned}
 \label{eq:bullet-lattices-maps}
\end{equation}
Then $(P,Q)$ is a skew-symmetrizable seed datum in the sense of
\cite[Section~4.1]{CMMM}.  Its $A$-type scattering diagram is the cluster
scattering diagram of $\Sigma_{\ell,m}$, and its chiral dual, in the terminology of \cite[Section~4.5]{CMMM}, is the seed
datum
\[ (Q^{\bullet},P^{\bullet}) \quad\text{on the dual lattices}\quad (N_{\ell,m}^{\bullet},M_{\ell,m}^{\bullet}). \]
If
\[ \mathbf B:=Q^{\mathsf T}P=(B_{\ell,m}^{\mathrm{uf}})^{\mathsf T}, \qquad \mathbf B^{\bullet}:=(Q^{\bullet})^{\mathsf T}P, \]
then
\begin{equation}
 \mathbf B=D\mathbf B^{\bullet},
 \qquad
 \mathbf B^{\bullet}_{uv}=\frac{b_{vu}}{d_u}
 =-\frac{b_{uv}}{d_v}.
 \label{eq:CMMM-skew-matrix}
\end{equation}
In particular, $\mathbf B^{\bullet}$ is skew-symmetric over $\mathbb Q$.
\end{proposition}

\begin{proof}
The lattices in \eqref{eq:bullet-lattices-maps} are mutually dual.  The map
$Q^{\bullet}$ is integral by construction.  For a mutable coordinate $v$,
the coefficient of $e_v$ in $P^{\bullet}(\epsilon_u)$ is
\[ d_ub_{uv}=-d_vb_{vu}, \]
which is divisible by $d_v$; frozen coordinates impose no further
divisibility condition.  Hence $P^{\bullet}$ takes values in
$M_{\ell,m}^{\bullet}$.  Equation \eqref{eq:CMMM-skew-matrix} follows from
\eqref{eq:symmetrizer}, and proves the seed-datum axioms.  With the mutable
indices ordered first, $P$ is the transpose of the extended exchange matrix
and $Q$ is the coordinate inclusion $\binom{I}{0}$.  Thus
\cite[Example~4.2]{CMMM} identifies $\mathcal D(P,Q)$ with the $A$-type
cluster scattering diagram for $\Sigma_{\ell,m}$.  The description of the
chiral dual is \cite[Section~4.5]{CMMM}.
\end{proof}

\begin{proposition}\label{prop:open-theta-basis}
For every $\ell,m\ge2$,
\[ \{\vartheta_g:g\in M_{\ell,m}\} \]
is a basis of $\cU_{\ell,m}^{\circ}$.
\end{proposition}

\begin{proof}
The terminal $c$-vectors of the reddening sequence in
Theorem~\ref{thm:reddening} are nonpositive.  By
\cite[Lemma~5.12]{GHKK}, they are the inward normals to the terminal cluster
chamber.  That chamber therefore meets the interior of the initial negative
chamber, so \cite[Lemma~8.29]{GHKK} gives a large cluster complex.  This is
the argument for $(2)\Rightarrow(3)$ in \cite[Corollary~8.30]{GHKK}; it uses
only the terminal sign condition.  Proposition~\ref{prop:folded-B} verifies
the convexity condition in \cite[Theorem~0.3(7)]{GHKK}.  Hence
\cite[Proposition~8.28 and Definition~0.6]{GHKK} identifies the middle,
upper, and canonical algebras.  The canonical basis is indexed by all
integral tropical points of the dual, which the initial seed identifies
with $M_{\ell,m}$.
\end{proof}

\begin{remark}[relation with Fock--Goncharov]
\label{rem:Fock-Goncharov-relation}
Two parts of the construction come from the work of Fock and Goncharov.  Their
local-system construction supplies the $\ell$-triangulated disk quivers and
the flip sequences used in Section~\ref{sec:reddening}
\cite[Section~10]{FockGoncharov}.  Their cluster ensembles consist of positive spaces of $A$- and
$X$-type, related by a map $p$; the duality conjectures use integral
tropical points
\cite[Sections~1--4]{FockGoncharovEnsembles}.  We use the disk quivers and flip sequences before folding and the
$A/X$ duality after folding; this does not require the invariant variety
to be a moduli space of local systems.  The orbit-size symmetrizer makes the chiral dual
of Proposition~\ref{prop:theta-seed-datum}, rather than the matrix transpose alone, the appropriate dual seed datum.  The cone $\mathscr C_{\ell,m}$ below is
the set of tropical parameters whose theta functions extend across the
frozen boundary.
\end{remark}

\subsection{The frozen boundary and its theta basis}

\begin{proposition}\label{prop:frozen-partial-compactification}
The scheme $\mathcal A_{\ell,m}$ is an integral partial compactification of
$\mathcal A_{\ell,m}^{\circ}$ with the following properties.
\begin{enumerate}
 \item Its global functions are the upper cluster algebra with polynomial
       frozen coefficients, while localization at all frozen variables
       gives the open upper cluster algebra:
       \begin{equation}
       \Gamma(\mathcal A_{\ell,m},\mathcal O)=\cU_{\ell,m},
       \qquad
       \Gamma(\mathcal A_{\ell,m}^{\circ},\mathcal O)
       =\cU_{\ell,m}^{\circ}.
       \label{eq:cluster-global-sections}
       \end{equation}
 \item For each frozen index $f$, the equations $z_f=0$ in the seed charts
       glue to an irreducible effective Cartier divisor $D_f$, with
       $\operatorname{ord}_{D_f}(z_f)=1$.  The complement of the open
       cluster variety is
       \begin{equation}
       \mathcal A_{\ell,m}\setminus\mathcal A_{\ell,m}^{\circ}
       =\bigcup_{f\in\mathscr F_{\ell,m}}D_f,
       \label{eq:frozen-boundary-union}
       \end{equation}
       and it has no other codimension-one components.
 \item The divisor $D=\sum_fD_f$ is anticanonical.  More precisely, the
       logarithmic volume forms
       \begin{equation}
       \Omega_t
       =\bigwedge_{u\in\mathscr M_{\ell,m}}d\log z_{u;t}
        \wedge
        \bigwedge_{f\in\mathscr F_{\ell,m}}d\log z_f
       \label{eq:cluster-volume-form}
       \end{equation}
       glue up to sign, and
       $\operatorname{div}(\Omega)=-\sum_fD_f$.
\end{enumerate}
Hence $(\mathcal A_{\ell,m},D)$ satisfies the hypotheses on
anticanonical partial compactifications in
\cite[Corollary~1.2]{CMMM}.
\end{proposition}

\begin{proof}
Put $p_t=\prod_{u\in\mathscr M_{\ell,m}}z_{u;t}$.  The cluster-algebra
identity of Theorem~\ref{thm:main-equality} gives
\[
 \cR_{\ell,m}[p_t^{-1}]
 =\kk[z_{u;t}^{\pm1}:u\in\mathscr M_{\ell,m}]
       [z_f:f\in\mathscr F_{\ell,m}].
\]
Indeed, all cluster coordinates lie in $\cR_{\ell,m}$, which is contained
in every chart ring.  Thus the charts are principal opens of
$\operatorname{Spec}\cR_{\ell,m}$ and their gluing is separated.  A global
regular function is exactly an element regular in every chart.  Their coordinate rings intersect in the upper cluster algebra with
polynomial frozen coefficients.  Inverting
all $z_f$ gives the corresponding open intersection, proving (1).

In any chart, $(z_f)$ is a height-one prime ideal and $z_f$ is a local
parameter.  Frozen coordinates are unchanged by mutable mutation.  On the
overlap for a mutation at $u$, exactly one exchange monomial can contain
$z_f$, because the single matrix entry $b_{uf}$ has one sign; after setting
$z_f=0$, the other monomial is generically nonzero.  Hence the two chart
pieces of $(z_f=0)$ have a common dense open subset and glue to the closure of
the irreducible initial-chart divisor $(z_f=0)$.  This proves that $D_f$ is an
irreducible prime Cartier divisor, while distinct frozen indices give distinct
generic points in the initial chart.  Localizing by $\prod_fz_f$ removes
exactly their union.  Thus every codimension-one component of the complement
is one of the $D_f$, proving (2).

A direct Jacobian calculation for one cluster mutation gives
$\Omega_{\mu_u(t)}=-\Omega_t$; hence the forms in
\eqref{eq:cluster-volume-form} glue up to sign.  At the generic point of
$D_f$, the factor $d\log z_f=dz_f/z_f$ has a simple pole and every other
coordinate in this chart is a unit.  Therefore
$\operatorname{ord}_{D_f}(\Omega)=-1$.  By
\eqref{eq:frozen-boundary-union} there is no other boundary divisor on
which a pole can occur, which proves (3).
\end{proof}

\begin{proposition}\label{prop:primitive-frozen-valuations}
For $f\in\mathscr F_{\ell,m}$, the divisorial valuation of $D_f$ is
represented in the initial seed coordinates by the primitive vector
\[ n_f=e_f^\vee\in N_{\ell,m}\subset N_{\ell,m}^{\bullet}. \]
In every seed obtained by mutable mutations, the valuation has value
$1$ on $z_f$ and value $0$ on every other frozen or mutable cluster
coordinate.
\end{proposition}

\begin{proof}
In the initial seed chart,
$\operatorname{val}_{D_f}(z^g)=g_f=\langle g,e_f^\vee\rangle$,
so the valuation vector is exactly $e_f^\vee$, not a nontrivial multiple.
It is primitive because $\operatorname{val}_{D_f}(z_f)=1$.

Assume inductively that all mutable cluster coordinates obtained so far
have $D_f$-valuation zero.  Mutation at $u$
has exchange relation
\[ z_uz_u' =\prod_vz_v^{[b_{uv}]_+} +\prod_vz_v^{[-b_{uv}]_+}. \]
The two monomials have $D_f$-valuations $[b_{uf}]_+$ and
$[-b_{uf}]_+$, so at least one has valuation zero.  Their sum is nonzero at
the generic point of $D_f$, and hence has valuation zero.  Since
$\operatorname{val}_{D_f}(z_u)=0$, also
$\operatorname{val}_{D_f}(z_u')=0$.  Induction proves the claim.
\end{proof}

For $f\in\mathscr F_{\ell,m}$, first define the boundary valuation function
on integral theta parameters by
\[ \nu_f:M_{\ell,m}\longrightarrow\mathbb R, \qquad \nu_f(g):=\operatorname{val}_{D_f}(\vartheta_g). \]
By valuative independence and
\cite[Theorem~1.1 and Corollary~1.2]{CMMM}, these functions extend to
integral piecewise-linear functions on $M_{\ell,m,\mathbb R}$.  Denote
the extensions by the same symbols and set
\begin{equation}
 \mathscr C_{\ell,m}
 :=\left\{g\in M_{\ell,m,\mathbb R}:
   \nu_f(g)\ge0\text{ for every }f\in\mathscr F_{\ell,m}\right\}.
 \label{eq:theta-cone}
\end{equation}
The same results show that \eqref{eq:theta-cone} is a polyhedron.
Each $\nu_f$ is an infimum of integral linear forms, so it is positively
homogeneous and concave in the ordinary sense; its superlevel sets are
convex.  Since the right-hand sides in
\eqref{eq:theta-cone} are zero, $\mathscr C_{\ell,m}$ is a rational
polyhedral cone.  After clearing denominators, choose an integral matrix
$H_{\ell,m}$ such that
\[ \mathscr C_{\ell,m} =\{g\in M_{\ell,m,\mathbb R}:H_{\ell,m}g\ge0\}. \]

For $n\in N_{\ell,m}^{\bullet}$ write
\[ \vartheta_n^{\chi}:=\Theta_{n,+}^{(Q^{\bullet},P^{\bullet})}. \]
Theta reciprocity \cite[Claim~4.16 and Theorem~5.19]{CMMM}, together with
Proposition~\ref{prop:primitive-frozen-valuations}, gives
\[ \nu_f(g)=\operatorname{val}_{g}(\vartheta_{n_f}^{\chi}) \]
whenever both sides are finite.  Thus, if $\vartheta_{n_f}^{\chi}=\sum_{a\in A_f}c_{f,a}X^a \ (c_{f,a}>0)$,
is a finite Laurent expansion in the initial chiral-dual chart, then
\[ \nu_f(g)=\min_{a\in A_f}\langle a,g\rangle. \]

\begin{remark}[choice of the chiral dual]
\label{rem:chiral-versus-Langlands}
The standard Langlands-dual pair is $(Q,-P)$.  The corresponding integral
seed-datum version of theta reciprocity requires both $D$ and
$\mathbf B^{\bullet}$ to be integral, and it uses the theta function
expanded from the \emph{negative} chamber:
\begin{equation}
 \operatorname{val}_{n}(\Theta_{g,+}^{(P,Q)})
 =\operatorname{val}_{g}(\Theta_{n,-}^{(Q,-P)}).
 \label{eq:Langlands-negative-chamber}
\end{equation}
By \cite[Corollary~4.20]{CMMM}, in the folded seed $D$ is integral,
but integrality of $\mathbf B^{\bullet}$ does not follow from
skew-symmetrizability and is not assumed.  Consequently all reciprocity and
formulas for the boundary theta functions in this paper use the chiral dual
$(Q^{\bullet},P^{\bullet})$.  Formula
\eqref{eq:Langlands-negative-chamber} is available in any individual case
where $\mathbf B^{\bullet}$ is integral, but the negative-chamber
convention must then be retained.
\end{remark}

\begin{theorem}[theta basis of the partial compactification]
\label{thm:theta-basis}
For every $\ell,m\ge2$,
\[ \left\{\vartheta_g: g\in\mathscr C_{\ell,m}\cap M_{\ell,m}\right\} \]
is a basis of $\cR_{\ell,m}$.
\end{theorem}

\begin{proof}
Proposition~\ref{prop:open-theta-basis} gives a theta basis on the open
cluster variety.  Corollary~1.2 of \cite{CMMM} and valuative independence show that
the members regular along every $D_f$ form a basis of the ring of global
functions on the partial compactification.  By definition, regularity is exactly the condition
$g\in\mathscr C_{\ell,m}$.  Finally,
\eqref{eq:cluster-global-sections} and Theorem~\ref{thm:main-equality}
identify the resulting algebra with $\cR_{\ell,m}$.
\end{proof}

\subsection{Weight fibers and plethysm}

Define the initial weight map, using the weights of
Proposition~\ref{prop:weights-count}, by
\[ \mathsf W_{\ell,m}:M_{\ell,m} \longrightarrow X^*(T_V)\oplus\mathbb Z^m, \qquad e_v\longmapsto\bigl(\wt_V(z_v),\deg_W(z_v)\bigr). \]
Exchange homogeneity implies
\[ \mathsf W_{\ell,m}(b_{u\bullet})=0 \qquad(u\in\mathscr M_{\ell,m}). \]
Thus $\widehat y_u=\prod_vz_v^{b_{uv}}$ has weight zero, and a pointed
theta expansion yields
\begin{equation}
 \wt(\vartheta_g)=\mathsf W_{\ell,m}g.
 \label{eq:theta-weight}
\end{equation}

For a dominant polynomial $\GL_\ell$-weight $\lambda$ and a weak composition
$\alpha\in\mathbb Z_{\ge0}^m$, set
\[ \mathscr P_{\ell,m}(\lambda;\alpha) :=\left\{g\in M_{\ell,m,\mathbb R}: H_{\ell,m}g\ge0, \quad \mathsf W_{\ell,m}g=(\lambda;\alpha) \right\}. \]

\begin{corollary}\label{cor:polyhedral-b}
For every $\ell,m\ge2$, we have
\begin{equation}
 b_{\alpha,(2)}^\lambda
 =\left|\mathscr P_{\ell,m}(\lambda;\alpha)\cap M_{\ell,m}\right|.
 \label{eq:polyhedral-b}
\end{equation}
\end{corollary}

\begin{proof}
By Theorem~\ref{thm:theta-basis} and \eqref{eq:theta-weight}, these
lattice points index a basis of the $(\lambda;\alpha)$-weight space of
$\cR_{\ell,m}$.  Its dimension is $b_{\alpha,(2)}^\lambda$ by
\eqref{eq:b-as-multiplicity}.  The set is finite because this
multihomogeneous component of the polynomial invariant ring is
finite-dimensional.
\end{proof}

\begin{theorem}[polyhedral formula for symmetric-square plethysm]
\label{thm:polyhedral-plethysm}
Let $\ell,m\ge2$, and let $\lambda,\mu$ be partitions with
$\htp(\lambda)\le\ell$ and $\htp(\mu)\le m$.  Then
\[  a_{\mu,(2)}^\lambda =\sum_{\sigma\in\frakS_m}\operatorname{sgn}(\sigma) \left| \mathscr P_{\ell,m} \bigl(\lambda;\alpha^\sigma(\mu)\bigr) \cap M_{\ell,m} \right|. \]
Here $\alpha^\sigma(\mu)_i=\mu_i-i+\sigma(i) \ (1\le i\le m)$,
and a summand is zero if $\alpha^\sigma(\mu)$ has a negative component.
\end{theorem}

\begin{proof}
Substitute \eqref{eq:polyhedral-b} into the alternating formula of
Theorem~\ref{thm:alternating-plethysm}; Theorem~\ref{thm:main-equality}
identifies the cluster algebra used in the count with the invariant ring.
\end{proof}

\subsection{Optimized boundary seeds}

The valuation-theoretic statements above hold for all $(\ell,m)$.  When $\ell$
is odd and $m$ is even, each boundary valuation is optimized in a seed
obtained by mutation.

Let $\Sigma'$ be a seed in the mutation class, with seed map $P'$.  For a
frozen index $f$, Proposition~\ref{prop:primitive-frozen-valuations} and
$P'^{\bullet}=P'D$ give
\[ \langle n_f,P'(\epsilon_u)\rangle=b'_{uf}. \]
Hence $\Sigma'$ is optimized for $n_f$ precisely when $b'_{uf}\ge0$ for
every mutable $u$.  In this case $\vartheta_{n_f}^{\chi}$ restricts to
$X^{n_f}$ on the corresponding chiral-dual torus by
\cite[Definition~9.1 and Lemma~9.3]{GHKK}.

\begin{proposition}\label{prop:optimized-boundary-seeds}
Assume that $\ell$ is odd and $m$ is even.  Every frozen boundary divisor
admits an optimized seed.  More precisely, for each frozen index $f$,
some sequence of mutations at mutable vertices gives $B_{\bullet f}=e_u$ for a mutable
index $u$.
\end{proposition}

\begin{proof}
Proposition~\ref{prop:boundary-singleton-isolation} gives explicit mutation
sequences that make each frozen column $e_u$ or $-e_u$.  For an outer
pair $f_i=[m;i,\ell-i]$, $f_{\ell-i}$ with $i<\ell/2$, the induction adds two
triangles at each step and gives the columns $e_u$ and $-e_u$ simultaneously.  Oddness of
$\ell$ ensures that the pairs exhaust the outer boundary.

It remains to change a column $-e_u$ to a standard basis column.  The resulting mutable principal matrix $B$ admits a reddening sequence by
Theorem~\ref{thm:reddening} and mutation invariance.  A reddening path ends
at a relabelled coframed seed; reversing the path and relabelling gives,
with frozen indices labelling columns, a path from $[B\mid-I]$ to a matrix
whose coefficient block is a permutation matrix
\cite[Lemma~2.2.1 and Theorem~2.2.4]{BHIT}.
The mutation rule for one frozen column is independent of all other frozen
columns.  The same path therefore sends the chosen $-e_u$ to $e_v$, regardless
of the other coefficients.  This gives an optimizing seed for each
frozen divisor; these seeds need not coincide.
\end{proof}

For odd $\ell$ and even $m$, put
\[ W_{\partial}^{\chi} :=\sum_{f\in\mathscr F_{\ell,m}}\vartheta_{n_f}^{\chi}. \]

\begin{corollary}\label{cor:boundary-cone-optimized-seeds}
Assume that $\ell$ is odd and $m$ is even.  For each frozen index $f$,
choose an optimizing mutation sequence and write the pullback of $X^{n_f}$ to
the initial chiral-dual chart as
\begin{equation}
 \vartheta_{n_f}^{\chi}=\sum_{a\in A_f}c_{f,a}X^a,
 \qquad c_{f,a}>0.
 \label{eq:boundary-theta-Laurent}
\end{equation}
Then
\begin{equation}
 \mathscr C_{\ell,m}
 =\left\{g:\langle a,g\rangle\ge0
   \text{ for every }f\in\mathscr F_{\ell,m},\ a\in A_f\right\}.
 \label{eq:optimized-boundary-cone}
\end{equation}
Clearing denominators and removing redundant inequalities gives an integral
inequality matrix $H_{\ell,m}$.
\end{corollary}

\begin{proof}
By the optimized-seed criterion above, the theta function is a monomial in
the optimized seed.  Pulling it back through the inverse chiral-dual
mutation sequence gives \eqref{eq:boundary-theta-Laurent}.  Theta reciprocity yields
\[ \nu_f(g)=\min_{a\in A_f}\langle a,g\rangle, \]
so regularity along every frozen divisor is equivalent to
\eqref{eq:optimized-boundary-cone}.  Equivalently,
\[ \mathscr C_{\ell,m} =\{g:(W_{\partial}^{\chi})^{\operatorname{trop}}(g)\ge0\}. \]
\end{proof}

\subsection{Degree-zero faces and fewer matrices}
\label{subsec:successive-matrix-faces}

Let $m\ge3$.  For $1\le s\le m$, let
\[ d_s:=\operatorname{pr}_s\circ\mathsf W_{\ell,m}:M_{\ell,m,\mathbb R}\longrightarrow\mathbb R, \]
where $\operatorname{pr}_s$ is projection to the $s$th $T_W$-degree.  For
an integral theta parameter $g$, equation~\eqref{eq:theta-weight}
identifies $d_s(g)$ with the degree of $\vartheta_g$ in the entries of
$C_s$.  If $2\le r<m$, put
\[ \mathscr C_{\ell,r}^{[m]}:=\mathscr C_{\ell,m}\cap\bigcap_{s=r+1}^m\ker(d_s), \qquad \delta_{r,m}:=\sum_{s=r+1}^m d_s. \]
The superscript $[m]$ records the ambient cone; set $\mathscr C_{\ell,m}^{[m]}:=\mathscr C_{\ell,m}$.

\begin{proposition}\label{prop:successive-matrix-face}
For every $\ell\ge2$, $m\ge3$, and $2\le r<m$, each $d_s$ is nonnegative
on $\mathscr C_{\ell,m}$, and
\[ \mathscr C_{\ell,r}^{[m]}=\mathscr C_{\ell,m}\cap\ker(\delta_{r,m}) \]
is a proper exposed rational polyhedral face of $\mathscr C_{\ell,m}$.  Under the
pullback $\cR_{\ell,r}\hookrightarrow\cR_{\ell,m}$ induced by the
projection $(C_1,\ldots,C_m)\mapsto(C_1,\ldots,C_r)$, the family
\[ \{\vartheta_g:g\in\mathscr C_{\ell,r}^{[m]}\cap M_{\ell,m}\} \]
is a basis of $\cR_{\ell,r}$.  Thus, for
$\alpha\in\mathbb Z_{\ge0}^r$ and a dominant polynomial
$\GL_\ell$-weight $\lambda$,
\[
 b_{\alpha,(2)}^\lambda
 =\left|\left\{g\in\mathscr C_{\ell,r}^{[m]}\cap M_{\ell,m}:
 \mathsf W_{\ell,m}g=
 (\lambda;\alpha,\underbrace{0,\ldots,0}_{m-r})\right\}\right|.
\]
If $H_{\ell,m}$ is an inequality matrix for $\mathscr C_{\ell,m}$, this
model is obtained by imposing
\[ H_{\ell,m}g\ge0,\qquad d_{r+1}(g)=\cdots=d_m(g)=0, \]
and restricting the weight map to the first $r$ matrix degrees.
\end{proposition}

\begin{proof}
Let $g\in\mathscr C_{\ell,m}\cap M_{\ell,m}$.  By
Theorem~\ref{thm:theta-basis}, $\vartheta_g$ is a nonzero polynomial
invariant in $\cR_{\ell,m}$, and
\eqref{eq:theta-weight} makes it homogeneous of multidegree
$(d_1(g),\ldots,d_m(g))$.  Hence $d_s(g)\ge0$ for every $s$.  A rational
polyhedral cone is generated by lattice vectors, so each $d_s$ is
nonnegative on all of $\mathscr C_{\ell,m}$.  It follows that
$\delta_{r,m}\ge0$ on the cone and that its zero locus there is precisely
$\mathscr C_{\ell,r}^{[m]}$.  This proves the face assertion.

The simultaneous degree-zero part of the polynomial ring on $m$ symmetric
matrices, with respect to $C_{r+1},\ldots,C_m$, is the polynomial ring on
the first $r$ matrices.  The congruence action preserves the full
multigrading, and therefore
\[ (\cR_{\ell,m})_{d_{r+1}=\cdots=d_m=0}=\cR_{\ell,r}. \]
The theta basis is homogeneous, so the members with these degrees zero form
a basis of this subalgebra.  Since $\Delta_m$ has positive
$\delta_{r,m}$-degree, some homogeneous theta basis element lies outside the
face; hence the face is proper.  The counting formula and the description by
restricted inequalities and weights follow.
\end{proof}

For fixed $m$, these faces form the chain
\[ \mathscr C_{\ell,2}^{[m]}\subset\mathscr C_{\ell,3}^{[m]}\subset\cdots\subset\mathscr C_{\ell,m-1}^{[m]}\subset\mathscr C_{\ell,m}. \]
For $2\le r<m$,
\[ \mathscr C_{\ell,r}^{[m]}=\mathscr C_{\ell,r+1}^{[m]}\cap\ker(d_{r+1}), \]
so each inclusion is an exposed face.  Taking $m=2n$ and $r=2n-1$ gives a
polyhedral theta model for $\cR_{\ell,2n-1}$ as a face of
$\mathscr C_{\ell,2n}$.  If $\ell$ is odd,
Corollary~\ref{cor:boundary-cone-optimized-seeds} supplies the explicit
matrix $H_{\ell,2n}$, which may then be restricted by $d_{2n}=0$.

\subsection{Descent to the ground field}

\begin{lemma}\label{lem:theta-characteristic-zero-descent}
Let $A_{\mathbb Q}$ be the polynomial ring in the matrix entries defining
$X_{\ell,m}$, and put
\[ R_{\mathbb Q}:=A_{\mathbb Q}^{U_{\ell,\mathbb Q}}. \]
For every algebraically closed field $K$ of characteristic zero, scalar
extension induces
\begin{equation}
 R_{\mathbb Q}\otimes_{\mathbb Q}K
 \simeq K[X_{\ell,m}]^{U_{\ell,K}}.
 \label{eq:invariant-base-change}
\end{equation}
Fix the normalizations of Theorem~\ref{thm:flagged-matrix-form} over
$k_0=\overline{\mathbb Q}$.  The finite theta functions then belong to this
common $k_0$-form of the cluster fraction field.  If they form a theta
basis over $\mathbb C$, the same indexed family is a basis over every
such $K$, after choosing compatible embeddings of $k_0$.
\end{lemma}

\begin{proof}
The group $U_{\ell,\mathbb Q}$ is generated by its simple-root
one-parameter subgroups.  On each polynomial multidegree, invariants are
the common kernel of finitely many rational linear derivations between
finite-dimensional spaces.  Flat scalar extension preserves this kernel,
proving \eqref{eq:invariant-base-change}.

The algebraic seed construction and its normalizations may be carried out
over the algebraically closed field $k_0$.  The scattering and broken-line
rules have integral coefficients, so every finite theta Laurent expansion
in these normalized cluster coordinates belongs to $k_0(X_{\ell,m})$.
If it is regular over $\mathbb C$, it belongs already to $k_0[X_{\ell,m}]$:
\[
 k_0[X_{\ell,m}]
 =\mathbb C[X_{\ell,m}]\cap k_0(X_{\ell,m})
 \quad\text{inside }\mathbb C(X_{\ell,m}).
\]
The same kernel argument gives invariance over $k_0$.  In each
multidegree, independence and spanning of the theta family descend from
$\mathbb C$ by faithful flatness and then extend to $K$.  Boundary
regularity is the same field-independent system of valuation inequalities.
This argument does not require the normalizing scalars to be rational.
\end{proof}

\medskip
\noindent\emph{Computational verification.}
The accompanying code reconstructs the extended exchange matrix from
Theorem~\ref{thm:folded-relations} and checks the row identities for the
mutation sequences in Appendix~\ref{app:boundary-optimization}, using
exact integer arithmetic.  Mutable variables precede frozen variables.
The code stores the weights as rows of
$W_{\rm code}=\mathsf W_{\ell,m}^{\mathsf T}$, so
\[
 B_{\ell,m}W_{\rm code}=0,
 \qquad \mathsf W_{\ell,m}g=(\lambda;\alpha).
\]
The code does not construct general reddening sequences or inequality
matrices.  To obtain the inequalities, one must first choose a reddening
sequence for the sign change in
Proposition~\ref{prop:optimized-boundary-seeds}, then apply
Corollary~\ref{cor:boundary-cone-optimized-seeds}.  The face models of
Proposition~\ref{prop:successive-matrix-face} use the ambient lattice and
the defining face equations; they are not identified here with the
initial coordinates of the smaller seed.

\appendix

\section{Transfer from the flagged Kronecker model}
\label{app:matrix-transfer}

We prove Proposition~\ref{prop:matrix-transfer} and identify the lifted
semi-invariants of \cite{FeiKroneckerII} with the matrix chamber functions of
Sections~\ref{sec:matrix-invariants} and \ref{sec:folding}.  We first take the
quotient by the groups acting on the flag arms, then evaluate the Schofield
presentations and their minimal lifts.

\subsection{The flagged representation space}

Let $\mathsf K_{\ell,m}$ be the quiver with vertex set
\[ \{-\ell,\ldots,-1,1,\ldots,\ell\}. \]
It consists of two oppositely oriented type-$A$ arms joined by $m$ arrows
from $-\ell$ to $\ell$, as shown in Figure~\ref{fig:flagged-Klm}.

\begin{figure}[ht]
\centering
\begin{tikzpicture}[
  >=Latex,
  qvertex/.style={circle,draw,fill=white,minimum size=2.7mm,inner sep=0pt},
  qlabel/.style={font=\small},
  qdim/.style={font=\scriptsize},
  every path/.style={line width=.45pt}
]
\node[qvertex] (n1) at (0,0) {};
\node[qvertex] (n2) at (1.15,0) {};
\node (ndots) at (2.25,0) {$\cdots$};
\node[qvertex] (nlm) at (3.55,0) {};
\node[qvertex] (nl) at (5.05,0) {};
\node[qvertex] (pl) at (7.55,0) {};
\node[qvertex] (plm) at (9.05,0) {};
\node (pdots) at (10.35,0) {$\cdots$};
\node[qvertex] (p2) at (11.45,0) {};
\node[qvertex] (p1) at (12.60,0) {};

\node[qlabel] at (0,.38) {$-1$};
\node[qlabel] at (1.15,.38) {$-2$};
\node[qlabel] at (3.55,.38) {$-(\ell-1)$};
\node[qlabel] at (5.05,.38) {$-\ell$};
\node[qlabel] at (7.55,.38) {$\ell$};
\node[qlabel] at (9.05,.38) {$\ell-1$};
\node[qlabel] at (11.45,.38) {$2$};
\node[qlabel] at (12.60,.38) {$1$};

\draw[->] (n1) -- (n2);
\draw[->] (n2) -- (ndots);
\draw[->] (ndots) -- (nlm);
\draw[->] (nlm) -- (nl);
\draw[->] (pl) -- (plm);
\draw[->] (plm) -- (pdots);
\draw[->] (pdots) -- (p2);
\draw[->] (p2) -- (p1);

\draw[->] (nl) to[bend left=27] node[above] {$a_1$} (pl);
\draw[->] (nl) -- node[above] {$\cdots$} (pl);
\draw[->] (nl) to[bend right=27] node[below] {$a_m$} (pl);

\node[qdim,anchor=east] at (-.35,-.52) {$\boldsymbol\beta_\ell:$};
\node[qdim] at (0,-.52) {$1$};
\node[qdim] at (1.15,-.52) {$2$};
\node[qdim] at (2.25,-.52) {$\cdots$};
\node[qdim] at (3.55,-.52) {$\ell-1$};
\node[qdim] at (5.05,-.52) {$\ell$};
\node[qdim] at (7.55,-.52) {$\ell$};
\node[qdim] at (9.05,-.52) {$\ell-1$};
\node[qdim] at (10.35,-.52) {$\cdots$};
\node[qdim] at (11.45,-.52) {$2$};
\node[qdim] at (12.60,-.52) {$1$};
\end{tikzpicture}
\caption{The flagged $m$-arrow Kronecker quiver $\mathsf K_{\ell,m}$ and
its standard dimension vector.  The central arrows are
$a_r:-\ell\to\ell$ for $1\le r\le m$.}
\label{fig:flagged-Klm}
\end{figure}
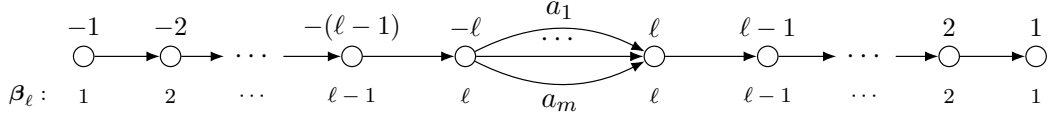

Thus $\boldsymbol\beta_\ell(\pm i)=i$.  Write $V_{\pm i}=\kk^i$.  The
representation space decomposes as
\[ \Rep_{\boldsymbol\beta_\ell}(\mathsf K_{\ell,m}) =\Rep_{\boldsymbol\beta_\ell}(A_\ell) \times \Mat_\ell^m \times \Rep_{\boldsymbol\beta_\ell}(A_\ell^\vee), \]
where the middle factor records the central maps $A_r=M(a_r)$.  The
semi-invariant ring is
\[ \SI_{\boldsymbol\beta_\ell}(\mathsf K_{\ell,m}) :=\kk[\Rep_{\boldsymbol\beta_\ell}(\mathsf K_{\ell,m})] ^{\prod_v\SL(V_v)}. \]

Let $G=\SL_\ell$.  Taking the quotient by the base-change groups at the internal vertices
of the negative and positive arms gives the affine closures
\[ \mathcal A^-:=\operatorname{Spec}\kk[G]^{U^-}, \qquad \mathcal A^+:=\operatorname{Spec}\kk[G]^U \]
of $U^-\backslash G$ and $G/U$, respectively.  The first space carries a
right $G$-action and the second a left $G$-action.

The first fundamental theorem for the two type-$A$ flag arms identifies
the quotient by the internal base-change groups with
$\mathcal A^-\times\Mat_\ell^m\times\mathcal A^+$; see
\cite{Grosshans,DerksenWeyman}.  The two terminal copies of $G$ act by
\[ (F^-,A_1,\ldots,A_m,F^+) \longmapsto (F^-g^{-1},\,gA_1h^{-1},\ldots,gA_mh^{-1},\,hF^+). \]

Taking invariants under the two terminal groups gives
\[ \SI_{\boldsymbol\beta_\ell}(\mathsf K_{\ell,m}) \cong \kk[\mathcal A^-\times\Mat_\ell^m\times\mathcal A^+]^{G\times G}. \]

\subsection{The transfer isomorphism}

Choose the standard points $F^-\in\mathcal A^-$ and
$F^+\in\mathcal A^+$ represented by the standard complete flags.  Their
stabilizers are $U^-$ and $U$.  Evaluation induces
\begin{equation}
 \operatorname{ev}_{F^-,F^+}:
 \kk[\mathcal A^-\times\Mat_\ell^m\times\mathcal A^+]^{G\times G}
 \longrightarrow
 \kk[\Mat_\ell^m]^{U^-\times U}.
 \label{eq:app-evaluation}
\end{equation}

\begin{proposition}\label{prop:app-transfer}
The map \eqref{eq:app-evaluation} is a multigraded algebra isomorphism.
Consequently,
\begin{equation}
 \SI_{\boldsymbol\beta_\ell}(\mathsf K_{\ell,m})
 \xrightarrow{\sim} \kk[\Mat_\ell^m]^{U^-\times U}
 \xrightarrow{\sim} \cT_{\ell,m}.
 \label{eq:app-transfer-final}
\end{equation}
The second arrow is induced by
$U_\ell\xrightarrow{\sim}U^-$, $u\mapsto u^{-T}$.
\end{proposition}

\begin{proof}
The first isomorphism is the two-sided Grosshans transfer for the two basic
affine spaces, evaluated at the standard flags; see \cite{Grosshans}.
Substituting $a=u^{-T}$ identifies the resulting $U^-\times U$-action with
\eqref{eq:two-sided-action}.  The central scaling torus and the two terminal
diagonal tori commute with these maps, so the isomorphism preserves the full
multigrading.
\end{proof}

In particular, the determinant of the $r$-th central arrow maps to
$\Delta_r=\det A_r$.

\subsection{Evaluation of the Schofield presentations}

We use Schofield's determinantal construction of quiver semi-invariants
\cite{SchofieldSemi}; the particular presentations and paths below are those
of \cite{FeiKroneckerII}.  On the open set where the central determinants are nonzero, that
construction adjoins formal inverses of the central arrows.  If $I=(r_1,\ldots,r_{2q+1})$ has odd length, the
corresponding path evaluates as
\[ A_I=A_{r_1}A_{r_2}^{-1}A_{r_3}\cdots A_{r_{2q}}^{-1}A_{r_{2q+1}}. \]
For $i+j=k$, the presentation used for the left triangle is
\[ f_{i,j}^{I,J}:P_k\longrightarrow P_{-i}\oplus P_{-j}, \]
and the reflected presentation reverses the paths and interchanges source
and target.

\begin{lemma}[matrix form of the Schofield semi-invariants]
\label{lem:app-Schofield-matrix}
After restricting both flag arms to the standard flags, the Schofield
semi-invariant of $f_{i,j}^{I,J}$ is, up to a nonzero scalar,
\begin{equation}
 \det
 \begin{bmatrix}
  P_iA_IJ_k\\[1mm]
  P_jA_JJ_k
 \end{bmatrix}.
 \label{eq:app-left-block}
\end{equation}
The reflected semi-invariant is, up to the corresponding scalar,
\begin{equation}
 \det\left[
  P_kA_{\overleftarrow I}J_i
  \ \middle|\
  P_kA_{\overleftarrow J}J_j
 \right].
 \label{eq:app-right-block}
\end{equation}
Reflection of the flagged quiver exchanges these two determinants and
becomes matrix transposition under \eqref{eq:app-transfer-final}.
\end{lemma}

\begin{proof}
On the standard negative arm, the unique path from the $i$-dimensional
vertex to the central vertex evaluates as $J_i$; on the standard positive
arm, the unique path from the central vertex to the $i$-dimensional vertex
evaluates as $P_i$.  By definition, the Schofield matrix is obtained from
the transpose of the presentation matrix by evaluating each path.  Taking its determinant gives \eqref{eq:app-left-block}.  Reflection
reverses all paths and interchanges source and target, giving
\eqref{eq:app-right-block}.
\end{proof}

Table~\ref{tab:flagged-matrix-dictionary} summarizes these identifications.
We use the notation of \cite{FeiKroneckerII}: $({}_{i,j}^{n})$ for the left
triangle and $({}_{i,j}^{n\vee})$ for its reflection.

\begin{table}[H]
\centering
\small
\renewcommand{\arraystretch}{1.35}
\begin{tabular}{@{}L{.20\textwidth}|L{.72\textwidth}@{}}
\textbf{Index in \cite{FeiKroneckerII}} & \textbf{presentation and image after standard-flag evaluation}\\ \hline
$({}_{i,j}^{n})$, $k=i+j$
 & $f_{i,j}^{n,\mathbf q_n}:P_k\to P_{-i}\oplus P_{-j}$, and
   \(\displaystyle
   s(f_{i,j}^{n,\mathbf q_n})\longmapsto
   \det\!\begin{bmatrix}P_iA_nJ_k\\ P_jA_{\mathbf q_n}J_k\end{bmatrix}
   =\phi_{i,j}^{(n),\mathrm L,\circ}.\)\\[2mm]
$({}_{i,j}^{n\vee})$, $k=i+j$
 & The reflected presentation $\check f_{i,j}^{n,\mathbf q_n}$ maps to
   \(\displaystyle
   \det\!\left[P_kA_nJ_i\ \middle|\
                    P_kA_{\overleftarrow{\mathbf q_n}}J_j\right]
   =\phi_{i,j}^{(n),\mathrm R,\circ}.\)\\[2mm]
$n$
 & The coefficient $\det M(a_n)$ maps to $\Delta_n=\det A_n$.
\end{tabular}
\caption{Labels from \cite{FeiKroneckerII}, presentations, and matrix determinants.}
\label{tab:flagged-matrix-dictionary}
\end{table}

\subsection{Minimal lift, boundary identifications, and corner corrections}

Definition~2.2 and Lemma~2.5 of \cite{FeiKroneckerII} specify the
minimal lift as a function, not just its divisor.  Under the
identifications in Table~\ref{tab:flagged-matrix-dictionary}, its image is
\begin{equation}
 \widetilde s_{i,j}^{n}\longmapsto
 \begin{cases}
  \phi_{i,0}^{(n),\mathrm L,\circ},&j=0,\\[1mm]
  \Pi_n\phi_{i,j}^{(n),\mathrm L,\circ},
    &j>0\text{ and }(i+j<\ell\text{ or }n\text{ even}),\\[1mm]
  \phi_{i,j}^{(n-1),\mathrm L},
    &j>0,\ i+j=\ell,\ n\text{ odd},
 \end{cases}
 \label{eq:app-minimal-lift-dictionary}
\end{equation}
with the transpose formula for $\widetilde{\check s}_{i,j}^{n}$.  Thus
\eqref{eq:app-minimal-lift-dictionary} is precisely
\eqref{eq:regular-left}--\eqref{eq:regular-right}.  The identifications of labels on the triangle boundaries are:
\begin{equation}
\begin{array}{c|c|c}
 \text{edge in \cite{FeiKroneckerII}}&\text{identification in \cite{FeiKroneckerII}}&\text{matrix notation}\\ \hline
 \text{common base}&({}_{i,0}^{n})=({}_{i,0}^{n\vee})
    &\det(P_iA_nJ_i)\\
 n\ge4\text{ even}&({}_{0,j}^{n})=({}_{0,j}^{n-1})
    &x_{0,j}^{(n)}=x_{0,j}^{(n-1)}\\
 n\ge3\text{ odd},\ i+j=\ell
    &({}_{i,j}^{n})=({}_{i,j}^{n-1})
    &x_{i,j}^{(n)}=x_{i,j}^{(n-1)}
\end{array}
\label{eq:app-edge-identification-table}
\end{equation}
These are exactly the boundary identifications in
\eqref{eq:even-identification}--\eqref{eq:odd-identification}.

The only lifted exchange relations requiring a new determinant arrow are
the three corner families in \cite[Corollaries~5.3--5.4]{FeiKroneckerII};
they are listed below.
\begin{equation}
\begin{array}{c|c|c}
 \text{corner in \cite{FeiKroneckerII}}&\text{coefficient arrow}&\text{matrix relation}\\ \hline
 ({}_{\ell-1,0}^{n})&({}_{\ell-1,0}^{n})\longrightarrow n
   &\Delta_n\text{ at the base corner}\\
 n\text{ even}:({}_{\ell-1,1}^{n}),({}_{\ell-1,1}^{n\vee})
   &n\longrightarrow\text{both}
   &c_{n;\ell-1,1}=\Delta_n\\
 n\ge3\text{ odd}:({}_{0,\ell-1}^{n}),({}_{0,\ell-1}^{n\vee})
   &n\longrightarrow\text{both}
   &\kappa_{n;\ell-1}=\Delta_n
\end{array}
\label{eq:app-corner-correction-table}
\end{equation}
For the base corner, substituting
the cases of \eqref{eq:app-minimal-lift-dictionary} explains the
change with parity: for odd $n$, the boundary variable is identified with
one in triangle $n-1$, so the determinant coefficient occurs in the
opposite exchange monomial.  No other determinant correction is present.

Lemma~\ref{lem:app-Schofield-matrix} and
Table~\ref{tab:flagged-matrix-dictionary} identify the unlifted functions.
Equation~\eqref{eq:app-minimal-lift-dictionary} gives the minimal lifts, while
\eqref{eq:app-edge-identification-table} and
\eqref{eq:app-corner-correction-table} give the boundary identifications and
arrows involving determinant vertices.  Thus the transfer isomorphism carries the full lifted
seed of \cite{FeiKroneckerII}, including all exchange relations, to the matrix
seed of Sections~\ref{sec:matrix-invariants} and \ref{sec:folding}; reflection
of the flagged presentation becomes matrix transposition.

\section{Transverse curves for the remaining coprimality cases}
\label{app:exceptional-coprimality}

The three families in \eqref{eq:exceptional-coprimality-families} admit curves
transverse to the zero loci of both $z_u$ and the exchange polynomial $N_u$. Matrices
not specified below are identities.  Smaller matrices are extended by an
identity block, and $\doteq$ denotes equality up to a nonzero scalar.

For a mutable index $u$, write
\[ N_u=M_{u,+}+M_{u,-}=z_uz_u'. \]

\begin{proposition}\label{prop:exceptional-specializations}
Every mutable variable in
\eqref{eq:exceptional-coprimality-families} admits a curve $\gamma:\mathbb A^1\to X_{\ell,m}$ satisfying
\[ z_u(t)=t\eta(t),\qquad N_u(t)=tq(t), \qquad \eta(0)q(0)\ne0. \]
\end{proposition}

\begin{proof}
The permutation matrices below are symmetric involutions.  They are
realized in an alternating word by taking one earlier factor equal to the
chosen involution and all other factors equal to the identity.

\smallskip
\noindent\emph{Base-edge variable for even $n$.}
Take
\[ C_n(t)=\begin{pmatrix}t&0\\0&1\end{pmatrix}\oplus I_{\ell-2}. \]
Then
\[ z_{1,0}^{(n)}=t,\qquad z_{2,0}^{(n)}\doteq t, \qquad z_{0,1}^{(n)}\doteq1, \qquad z_{1,1}^{(n)}=0. \]
The base-edge numerator has order one.  This includes $\ell=2$, where
$z_{2,0}^{(n)}=\Delta_n=t$.

\smallskip
\noindent\emph{Base-edge variable for odd $n$.}
Let
\[ G=\begin{pmatrix}0&1\\1&0\end{pmatrix}\oplus I_{\ell-2}, \qquad C_n(t)=\begin{pmatrix}t&1\\1&1\end{pmatrix}\oplus I_{\ell-2}, \qquad D=t-1. \]
Choose the preceding word so that $A_{\mathbf q_n}=GC_n$.  For
$\ell\ge3$,
\[ z_{1,0}^{(n)}=t, \qquad z_{2,0}^{(n)}\doteq D, \qquad z_{0,1}^{(n)}\doteq1, \qquad z_{1,1}^{(n)}\doteq D. \]
Thus $N_{1,0}=aD+bD^2$ with $a,b\ne0$.  Polynomial regularity gives
$t\mid N_{1,0}$, so $a=b$ and $N_{1,0}=aDt$.  For $\ell=2$, the same
argument applied to the corner numerator $aD+b$ gives $N_{1,0}=at$.
Interchanging $C_1$ and $C_2$ in the even $n=2$ curve treats
$z_{0,1}^{(2)}$.

\smallskip
\noindent\emph{Interior variable for even $n$.}
For $z_{1,2}^{(n)}$, take
\[
 C_n(t)=
 \begin{pmatrix}
  0&1&t&0\\
  1&0&0&1\\
  t&0&0&0\\
  0&1&0&0
 \end{pmatrix}\oplus I_{\ell-4}.
\]
The relevant chamber determinants are
\[ z_{1,2}^{(n)}\doteq t, \qquad (z_{1,1},z_{2,2},z_{0,3};z_{1,3},z_{0,2},z_{2,1}) \doteq(-1,t,1;0,1,0). \]
The two terms in \eqref{eq:interior} therefore have sum of order one.

\smallskip
\noindent\emph{Interior variable for odd $n$ when $\ell=4$.}
Let $G$ exchange $e_1$ with $e_3$ and $e_2$ with $e_4$.  Choose the preceding
matrices so that $C_{n-1}=I_4$ and $A_{\mathbf q_{n-1}}=G$; hence
$A_{\mathbf q_n}=GC_n$.  Put
\[
 C_n(t)=
 \begin{pmatrix}
  1&0&0&0\\
  0&0&1&0\\
  0&1&0&t\\
  0&0&t&1
 \end{pmatrix}.
\]
Then
\[ z_{1,2}^{(n)}\doteq t, \qquad (z_{1,1},z_{2,2},z_{0,3};z_{1,3},z_{0,2},z_{2,1}) \doteq(1,1,t;0,0,-1). \]
Thus the first term in \eqref{eq:interior} has order one and the second
vanishes.

\smallskip
\noindent\emph{Interior variable for odd $n$ when $\ell\ge5$.}
Let $G$ exchange $e_1$ with $e_4$ and $e_2$ with $e_5$,
and put
\[
 C_n(t)=
 \begin{pmatrix}
  0&-1&0&-1&0\\
  -1&1&0&0&1\\
  0&0&-1&0&t\\
  -1&0&0&0&0\\
  0&1&t&0&0
 \end{pmatrix}\oplus I_{\ell-5}.
\]
With $A_{\mathbf q_n}=GC_n$,
\[ z_{1,2}^{(n)}\doteq-t, \qquad (z_{1,1},z_{2,2},z_{0,3};z_{1,3},z_{0,2},z_{2,1}) \doteq(-1,t,1;1,-1,0). \]
Again the numerator in \eqref{eq:interior} has order one.
\end{proof}

\section{Row formulas for a folded half-diamond}
\label{app:triangular-mutation}

We prove the row formulas for the triangular mutation sequence used in
Lemma~\ref{lem:folded-half-diamond-removal}.  Only the mutable principal
matrix enters the argument; determinant coefficients and frozen boundary
columns are omitted.

Let
\[ \mathsf T_N=\{(a,b)\in\mathbb Z_{\ge0}^2:1\le a+b\le N\}, \qquad \mathsf T_N^+=\{(a,b)\in\mathsf T_N:a>0\}. \]
For a label $(a,b)$, let $e_{a,b}$ be the corresponding standard basis row
vector, and set $e_{a,b}=0$ when the label lies outside $\mathsf T_N$.
Consider the extended matrix with row indices $\mathsf T_N^+$ and column
indices $\mathsf T_N$.  The row indexed by $(a,b)$ is
\begin{align}
 r_{a,0}
 &=e_{a+1,0}+2e_{a-1,1}-e_{a-1,0}-2e_{a,1},
 \label{eq:appendix-base-row}\\
 r_{a,b}
 &=e_{a,b-1}+e_{a+1,b}+e_{a-1,b+1}
   -e_{a,b+1}-e_{a-1,b}-e_{a+1,b-1}
 \quad(b>0).
 \label{eq:appendix-interior-row}
\end{align}
These are the mutable rows of an exposed even folded half-diamond.
Replacing the entire matrix by its negative only reverses every source
and sink, so the chosen orientation causes no loss of generality.

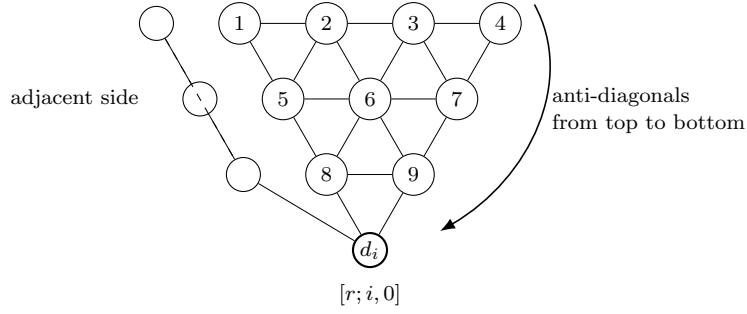
\begin{figure}[H]
\centering
\begin{tikzpicture}[
  every node/.style={font=\scriptsize},
  mutated/.style={circle,draw,minimum size=5.5mm,inner sep=0pt},
  plain/.style={circle,draw,minimum size=4.5mm,inner sep=0pt},
  edge/.style={line width=.35pt}
]
  \node[mutated] (v30) at (0,3) {$1$};
  \node[mutated] (v31) at (1.15,3) {$2$};
  \node[mutated] (v32) at (2.30,3) {$3$};
  \node[mutated] (v33) at (3.45,3) {$4$};
  \node[mutated] (v20) at (.575,2) {$5$};
  \node[mutated] (v21) at (1.725,2) {$6$};
  \node[mutated] (v22) at (2.875,2) {$7$};
  \node[mutated] (v10) at (1.15,1) {$8$};
  \node[mutated] (v11) at (2.30,1) {$9$};
  \node[plain,line width=.8pt] (d) at (1.725,0) {$d_i$};
  \node[plain] (p) at (.05,1) {};
  \node[plain] (c) at (-.525,2) {};
  \node[plain] (t) at (-1.10,3) {};

  \foreach \x/\y in {v30/v31,v31/v32,v32/v33,
                       v20/v21,v21/v22,v10/v11,
                       v30/v20,v31/v20,v31/v21,v32/v21,
                       v32/v22,v33/v22,v20/v10,v21/v10,
                       v21/v11,v22/v11,v10/d,v11/d,
                       t/c,c/p,p/d}
    \draw[edge] (\x)--(\y);
  \draw[dashed] (t)--(p);
  \node[anchor=east] at (-1.20,2) {adjacent side};
  \draw[-{Latex[length=2mm]},line width=.6pt]
    (3.9,3.25) .. controls (4.45,2.25) and (3.85,.95) .. (2.65,.25);
  \node[anchor=west,align=left] at (4.0,1.85)
    {anti-diagonals\\from top to bottom};
  \node[below=2pt of d] {$[r;i,0]$};
\end{tikzpicture}
\caption{The anti-diagonal mutation order when $N-i=3$.  The numbered
vertices are $v_{q,u}=(i+u,q-u)$ in mutation order; the numbers do not
indicate arrow orientation.}
\label{fig:half-diamond-mutation-order}
\end{figure}

Fix $1\le i<N$, delete the base vertices
$(1,0),\ldots,(i-1,0)$, and put
\[ h:=N-i, \qquad c:=(i-1,h+1), \]
\[ v_{q,u}:=(i+u,q-u) \qquad(1\le q\le h,\ 0\le u\le q). \]
The mutation sequence \eqref{eq:positive-removal-sequence}, with the triangle
index suppressed, is
\[ (v_{h,0},v_{h,1},\ldots,v_{h,h}; v_{h-1,0},\ldots,v_{h-1,h-1};\ldots; v_{1,0},v_{1,1}). \]

\begin{lemma}\label{lem:anti-diagonal-row-formula}
Immediately before mutation at $v_{q,u}$, the row $r_{v_{q,u}}$ is given
below.  In the cases $0<u<q$, put $a=i+u$ and $b=q-u$.  For the top anti-diagonal $q=h$,
\begin{equation}
 r_{v_{h,u}}=
 \begin{cases}
  -e_{i-1,h}+e_c+e_{i,h-1}-e_{i+1,h-1},&u=0,\\
  e_c-e_{a-1,b+1}+e_{a,b-1}-e_{a+1,b-1},&0<u<h,\\
  2e_c+e_{N-1,0}-2e_{N-1,1},&u=h.
 \end{cases}
 \label{eq:top-mutation-rows}
\end{equation}
For $1\le q<h$,
\begin{equation}
 r_{v_{q,u}}=
 \begin{cases}
  -e_{i-1,q}+e_{i,q-1}+e_{i,q+1}-e_{i+1,q-1},&u=0,\\
  -e_{a-1,b+1}+e_{a,b-1}+e_{a,b+1}-e_{a+1,b-1},&0<u<q,\\
  e_{i+q-1,0}-2e_{i+q-1,1}+e_{i+q+1,0},&u=q.
 \end{cases}
 \label{eq:lower-mutation-rows}
\end{equation}
All other entries are zero.
\end{lemma}

\begin{proof}
We use the extended-matrix mutation formula
\begin{equation}
 b'_{xj}=b_{xj}+[b_{xv}]_+[b_{vj}]_+
                 -[-b_{xv}]_+[-b_{vj}]_+
 \qquad(x,j\ne v),
 \label{eq:appendix-mutation-formula}
\end{equation}
together with $b'_{xv}=-b_{xv}$.  We induct along the mutation sequence, retaining every column,
including those on the adjacent unmutated side.

The $u=0$ case of \eqref{eq:top-mutation-rows} is
\eqref{eq:appendix-interior-row} for $(a,b)=(i,h)$: the two terms outside
the triangle vanish because $i+h=N$.  For the vertex $v=(a,b)=v_{q,u}$ being mutated, with $b>0$, set
\[ R:=(a+1,b-1),\qquad D:=(a,b-1). \]
Among vertices mutated later, $v$ is adjacent only to $R$ and $D$.
Thus mutation at $v$ leaves the rows indexed by all other later vertices
unchanged.  It suffices to compute $r_R$ and $r_D$, treating separately
the case where $R$ is a base vertex.

\smallskip
\noindent\emph{The right neighbor.}
Assume first that $b>1$.  Immediately before $\mu_v$, the row of
$R=(a+1,b-1)$ is
\begin{equation}
 r_R=
 \begin{cases}
  e_v-e_{a,b-1}+e_{a+1,b-2}-e_{a+2,b-2},&q=h,\\
  e_v-e_{a,b-1}-e_{a,b+1}
      +e_{a+1,b-2}+e_{a+1,b}-e_{a+2,b-2},&q<h.
 \end{cases}
 \label{eq:right-intermediate-row}
\end{equation}
For $q=h$ this is the original row
\eqref{eq:appendix-interior-row}; for $q<h$ it is the intermediate row obtained after mutating
the vertex $(a+1,b)$ directly above $R$, as recorded in the
calculation for the lower neighbor below.  Since $b_{Rv}=1$, substitution of
\eqref{eq:top-mutation-rows} or \eqref{eq:lower-mutation-rows} into
\eqref{eq:appendix-mutation-formula} cancels the negative terms involving
$a,b$ in \eqref{eq:right-intermediate-row}, changes $e_v$ to $-e_v$, and gives
the $0<u<q$ case of \eqref{eq:top-mutation-rows} or
\eqref{eq:lower-mutation-rows} at the next vertex.  Thus the induction advances from
$v_{q,u}$ to $v_{q,u+1}$.

If $b=1$, then $R=(a+1,0)$ is a transpose-fixed base vertex and
$b_{Rv}=2$.  Before $\mu_v$ its row is
\[
 r_R=
 \begin{cases}
  -e_{a,0}+2e_v,&q=h,\\
  -e_{a,0}+2e_v-2e_{a,2}+e_{a+2,0},&q<h.
 \end{cases}
\]
Formula \eqref{eq:appendix-mutation-formula} adds twice the positive part
of $r_v$ and reverses the $v$-entry.  It gives respectively
\[ 2e_c+e_{a,0}-2e_v, \qquad e_{a,0}-2e_v+e_{a+2,0}, \]
which are the $u=h$ and $u=q$ cases of
\eqref{eq:top-mutation-rows} and \eqref{eq:lower-mutation-rows}, respectively.

\smallskip
\noindent\emph{The lower neighbor.}
Suppose $b>1$.  If $u=0$, a direct substitution in
\eqref{eq:appendix-mutation-formula} gives
\[ r_D'=-e_{i-1,q-1}+e_{i,q-2}+e_v-e_{i+1,q-2}, \]
which is the $u=0$ case of \eqref{eq:lower-mutation-rows} for the next
anti-diagonal.  If $u>0$, the same calculation gives the intermediate row
\begin{equation}
 \overline r_D=
 -e_{a-1,b-1}+e_{a-1,b}-e_{a-1,b+1}
 +e_{a,b-2}+e_{a,b}-e_{a+1,b-2}.
 \label{eq:intermediate-lower-row}
\end{equation}
Before $D$ is reached in the next anti-diagonal, its left predecessor
$P=(a-1,b)$ is mutated.  Since $b_{DP}=1$, the terms
$e_{a-1,b-1}+e_{a-1,b+1}$ in the positive part of $r_P$ cancel the two
negative terms in \eqref{eq:intermediate-lower-row}.  Mutation at $P$
also reverses the $P$-entry, leaving
\[ -e_{a-1,b}+e_{a,b-2}+e_{a,b}-e_{a+1,b-2}, \]
which is the $0<u<q$ case of \eqref{eq:lower-mutation-rows} for $D$.
The $q<h$ case of \eqref{eq:right-intermediate-row} follows because
$R$ is the lower neighbor of $(a+1,b)$ in the preceding anti-diagonal.

For the last two vertices of an anti-diagonal, let
$v=(a,1)$ and $R=(a+1,0)$.  After the consecutive mutations
$\mu_v\mu_R$, the row of $D=(a,0)$ is
\begin{equation}
 \overline r_D=
 -e_{a-1,0}+2e_{a-1,1}-2e_{a-1,2}+e_{a+1,0}.
 \label{eq:intermediate-lower-base-row}
\end{equation}
On the next anti-diagonal, the last interior vertex
$P=(a-1,1)$ is mutated immediately before $D$.  Here $b_{DP}=2$, and the
positive part of $r_P$ relevant to
\eqref{eq:intermediate-lower-base-row} is
$e_{a-1,0}+e_{a-1,2}$.  Hence mutation at $P$ turns
\eqref{eq:intermediate-lower-base-row} into
$e_{a-1,0}-2e_{a-1,1}+e_{a+1,0}$,
the $u=q$ case of \eqref{eq:lower-mutation-rows}.

These calculations cover all rows indexed by later vertices that each
mutation changes.  Starting from $v_{h,0}$, they determine the rows on
the top anti-diagonal and then on each lower anti-diagonal.  This proves
\eqref{eq:top-mutation-rows}--\eqref{eq:lower-mutation-rows} by induction.
\end{proof}

\begin{proposition}\label{prop:triangular-mutation-sequence}
After deleting $(1,0),\ldots,(i-1,0)$ and applying
$\mathbf p_i^+$, the row of $(i,0)$ is
\begin{equation}
 r_{i,0}=e_{i+1,0}
 \qquad(1\le i<N).
 \label{eq:appendix-final-base-row}
\end{equation}
For $i=N$, the sequence is empty and, after the previous base vertices have
been deleted,
\begin{equation}
 r_{N,0}=2e_{N-1,1}.
 \label{eq:appendix-terminal-base-row}
\end{equation}
The same conclusions hold after deleting any additional unmutated boundary
vertices.  After deleting $(i,0)$, applying the mutations in reverse order
restores the original matrix on all surviving vertices.
\end{proposition}

\begin{proof}
Assume $i<N$.  Before the last anti-diagonal is reached, the row of
$d=(i,0)$ is unchanged, because none of the vertices with $q>1$ is adjacent
to it.  Since $(i-1,0)$ has already been deleted,
\begin{equation}
 r_d=2e_{i-1,1}-2e_{i,1}+e_{i+1,0}.
 \label{eq:appendix-tracked-row}
\end{equation}
Put $x=v_{1,0}=(i,1)$ and $y=v_{1,1}=(i+1,0)$.
Lemma~\ref{lem:anti-diagonal-row-formula} gives
\[ r_x=-e_{i-1,1}+e_d+t-e_y, \]
where $t=e_{i-1,2}$ if $h=1$ and $t=e_{i,2}$ if $h>1$; the term $t$ is
absent when its label is absent.  Since $b_{dx}=-2$, mutation at $x$
cancels the $2e_{i-1,1}$ term in
\eqref{eq:appendix-tracked-row}, changes the sign of the $x$-entry, and
gives
$r_d'=2e_x-e_y$.
Immediately before mutation at $y$, its row has the form
\[ r_y=e_d-2e_x+s, \]
where $s=2e_c$ if $h=1$ and $s=e_{i+2,0}$ if $h>1$.  In either case all
entries of $s$ are nonnegative.  Since $b_{dy}=-1$, mutation at $y$ cancels
$2e_x$ and reverses the $y$-entry.  Thus the row of $d$ passes through the
following three values:
\[ \underbrace{2e_{i-1,1}-2e_x+e_y}_{\text{before the last anti-diagonal}} \xrightarrow{\ \mu_x\ } \underbrace{2e_x-e_y}_{\text{after $\mu_x$}} \xrightarrow{\ \mu_y\ } \underbrace{e_y}_{\text{$d$ is a source}}. \]
This proves \eqref{eq:appendix-final-base-row}.

For $i=N$, formula \eqref{eq:appendix-base-row} has only the term
$2e_{N-1,1}$ after $(N-1,0)$ is deleted, which proves
\eqref{eq:appendix-terminal-base-row}.  Passing to a principal submatrix commutes with mutations at its retained
vertices.  Deleting any additional boundary vertices not in the sequence
therefore removes the corresponding coordinates from these identities.  The same commutation
shows that, after $d$ is deleted, the reverse sequence is the inverse of the
forward sequence on the restricted matrix.
\end{proof}

\section{Optimizing the frozen boundary}
\label{app:boundary-optimization}

We construct the mutation sequences used in
Proposition~\ref{prop:optimized-boundary-seeds}.  Throughout, $\ell\ge3$ is
odd, $m\ge2$ is even, and mutations are performed in the order written.

\subsection{The triangular chain and mutation sequences}

In every triangle with odd index $r=3,5,\ldots,m-1$, apply the hive twist
\begin{equation}
 \mathbf T_\ell^{(r)}
 =\bigl([r;a,b]\bigr)_{{\substack{
 h=1,\ldots,\ell-2;\ b=\ell-2,\ldots,h;\\
 a=\ell-b-1,\ldots,1}}},
 \qquad
 \mathbf T_{\ell,m}=\mathbf T_\ell^{(3)}\mathbf T_\ell^{(5)}
 \cdots\mathbf T_\ell^{(m-1)}.
 \label{eq:optimizer-twist}
\end{equation}
The indices run through the nested ranges in the displayed order.  This is the hive
sequence of \cite[Lemma~3.5]{FeiWeymanExtension}.  The interiors of the
two unfolded triangles have no arrows between them, so the paired
mutations fold by Lemma~\ref{lem:separated-pair}.  In the resulting seed,
retain the labels in even-indexed triangles and all interior and base
labels.  On the other boundaries, use
\[
 [r;0,k]_{\rm tw}=[r-1;k,\ell-k],\qquad
 [r;k,\ell-k]_{\rm tw}=[r;0,k]\quad(r\text{ odd}).
\]
Omitting the subscript, all consecutive triangles are now glued by
$[r;0,k]=[r-1;k,\ell-k]$.

For the outer-boundary calculation, delete the frozen columns indexed by
determinants.  Adjoin the frozen rows using the orbit-size symmetrizer,
and set $b_{f_kf_{k+1}}=-1$ for $f_k=[m;k,\ell-k]$.
These auxiliary frozen--frozen entries do not affect any mutable row.
Let $J$ be the full index set of this square matrix and write
$r_v=(b_{vw})_{w\in J}$.  All row identities retain every coordinate,
including those at unmutated boundary vertices; $e_v$ denotes a standard
basis row vector.  In column identities $B_{\bullet f}=\pm e_u$, only
mutable indices are retained and $e_u$ is a standard basis column vector.
Suppressing the triangle index, we have
\begin{align}
 r_{a,0}&=e_{a+1,0}+2e_{a-1,1}-e_{a-1,0}-2e_{a,1},\notag\\
 r_{a,b}&=e_{a,b-1}+e_{a+1,b}+e_{a-1,b+1}
          -e_{a,b+1}-e_{a-1,b}-e_{a+1,b-1}\quad(a,b>0,\ a+b<\ell).
 \label{eq:optimizer-triangle-rows}
\end{align}
A basis vector $e_v$ is set to zero when its label lies outside the
triangle or at a deleted corner.
For a vertex on the outer right boundary, use the second formula with
the same zero convention.  On the left boundary of the first triangle, the row is
$r_{[2;0,k]}=e_{[2;0,k-1]}+2e_{[2;1,k]}
-e_{[2;0,k+1]}-2e_{[2;1,k-1]}$.
At a common boundary $L_k=[r;0,k]$, the row is
\begin{equation}
 r_{L_k}=r_{L_k}^{\rm old}
       +e_{[r;1,k]}-e_{[r;1,k-1]},
 \label{eq:optimizer-seam-row}
\end{equation}
where $r_{L_k}^{\rm old}$ is the row for the preceding chain, completed
at its right boundary in the same way.  Thus the added terms have coefficient one, unlike those at the first
left boundary.
The hive twist, including its boundary arrows, and the identifications in
Theorem~\ref{thm:folded-relations} give these formulas.

Fix $1\le i<\ell/2$ and put $j=\ell-i$.  Define the mutation words
\begin{align}
 \mathbf Q_{3,1}&=((1,1),(1,0),(0,1),(1,1)),\notag\\
 \mathbf Q_{2i+1,i}
 &=\bigl((i,i),\ (i,t)_{t=0}^{i-1},\
               (t,i-t)_{t=0}^{i-1},\ (t,i)_{t=1}^{i-1}\bigr)
       \quad(i\ge2),\notag\\
 \mathbf Q_{\ell,i}
 &=\bigl((i,j-1),\ \mathbf Q_{\ell-2,i},\
           (i,j-2),(j-1,i),(j-2,i)\bigr)
       \quad(\ell>2i+1),\label{eq:optimizer-Q}\\
 \mathbf R_{\ell,i}
 &=\bigl((i,t)_{t=0}^{j-1},\
          (t,i-t)_{t=1}^{i-1},\ (t,j)_{t=1}^{i-1}\bigr).
 \label{eq:optimizer-R}
\end{align}
The subscripts increase in mutation order; empty sequences are omitted.
The sum of the two coordinates of each mutated vertex is less than $\ell$.
The only left-boundary vertex in $\mathbf Q_{\ell,i}$ is $(0,i)$, occurring
once; $\mathbf R_{\ell,i}$ uses only base and interior vertices.

\subsection{Row identities}

\begin{lemma}\label{lem:optimizer-transfer}
In the two-matrix triangle, put $f=(i,j)$, $h=(j,i)$,
$f_-=(i-1,j+1)$, and $h_+=(j+1,i-1)$.  After
$\mathbf Q_{\ell,i}$, for a mutable vertex $u$,
\begin{equation}
 r_f=-e_u+e_h+e_{f_-},\qquad
 r_h=e_u-e_f-e_{h_+}.
 \label{eq:optimizer-closed-row}
\end{equation}

For $m\ge4$, denote the vertices of the last two triangles by
$x_{a,b}=[m;a,b]$ and $y_{a,b}=[m-1;a,b]$, and let ``old'' denote the chain
through triangle $m-2$.
Put
\[
 \begin{gathered}
 F=x_{i,j},\quad H=x_{j,i},\quad
 F_-=x_{i-1,j+1},\quad H_+=x_{j+1,i-1},\\
 b=y_{0,i},\quad d=y_{0,j},\quad p=y_{0,i-1},\quad q=y_{0,j+1}.
 \end{gathered}
\]
Then $\mathbf Q_{\ell,i}^{(m)}\mathbf R_{\ell,i}^{(m-1)}$ gives
\begin{equation}
 \begin{aligned}
 r_F&=e_b-e_\gamma+e_{F_-},&
 r_H&=e_\delta-e_d-e_{H_+},\\
 r_b&=r_b^{\rm old}-e_p+e_\gamma-e_F,&
 r_d&=r_d^{\rm old}+e_q-e_\delta+e_H,
 \end{aligned}
 \label{eq:optimizer-transfer-rows}
\end{equation}
where $\gamma,\delta$ belong to the two new triangles.  Every old mutable
row is unchanged, including its zero entries at all new vertices.
\end{lemma}

\begin{proof}
We induct along the sequences, retaining all unmutated boundary
coordinates.  First suppose that triangle $x$ shares its left boundary
$s_k=x_{0,k}=y_{k,\ell-k}$ with $y$.  Set $A=y_{i,j-1}$ and $B=y_{i-1,j}$.
The row at $s_i$ is
\[
 r_{s_i}=e_{s_{i-1}}-e_{s_{i+1}}+e_{x_{1,i}}-e_{x_{1,i-1}}
          +e_A-e_B.
\]
In this case, applying $\mathbf Q_{\ell,i}$ gives the following
analogue of \eqref{eq:optimizer-closed-row}:
\begin{equation}
 \begin{aligned}
 r_F&=e_A-e_\alpha+e_{F_-},&
 r_H&=e_\beta-e_B-e_{H_+},\\
 r_A&=r_A^0+e_{s_i}-e_{s_{i+1}}-e_B+e_\alpha-e_F,&
 r_B&=r_B^0+e_{s_{i-1}}-e_{s_i}+e_A-e_\beta+e_H.
 \end{aligned}
 \label{eq:optimizer-open-row}
\end{equation}
The superscript $0$ denotes the row before the sequence is applied, and
$\alpha,\beta$ are vertices of $x$.  The rows at every other base or interior vertex
of the preceding triangle are unchanged.

To verify these identities and \eqref{eq:optimizer-closed-row}, first take $\ell=2i+1$, $i\ge2$,
and write $c=x_{i,i}$, $X_t=x_{i,t}$, $Y_t=x_{t,i-t}$, and $Z_t=x_{t,i}$.
After mutation at $c$ and the $X_t$,
\[
 r_F=e_{Y_{i-1}}-e_{X_{i-1}}+e_{F_-},\qquad
 r_H=e_c-e_{Z_{i-1}}-e_{H_+}.
\]
Along each sequence, the matrix mutation formula cancels the two side
entries contributed by the preceding mutation.  Mutation along the $Y_t$
replaces $e_{Y_{i-1}}-e_{X_{i-1}}$ in $r_F$ by $e_{Z_1}-e_{Y_{i-1}}$ in the two-matrix case,
and by $e_A-e_{Y_{i-1}}$ in the shared-boundary case.  Immediately before mutation at $Z_{i-1}$, its row is respectively
\[
 e_{Z_{i-2}}-e_c-e_F+e_H,
 \qquad e_{Z_{i-2}}-e_c-e_B+e_H.
\]
For $i=2$, interpret $Z_0$ here as $Y_{i-1}$ in the two-matrix case and as
$s_i$ in the shared-boundary case.  These give \eqref{eq:optimizer-closed-row} with
$u=Z_{i-1}$, and \eqref{eq:optimizer-open-row} with
$\alpha=Y_{i-1}$, $\beta=Z_{i-1}$; successive terms in $r_A,r_B$ cancel along the same sequences.  For $i=1$, the four mutations of $\mathbf Q_{3,1}$ give
$u=x_{1,1}$ in the two-matrix case and
$(\alpha,\beta)=(x_{1,1},s_1)$ in the shared-boundary case.

For the induction on $\ell$, write
$a=x_{i,j-1}$, $b_0=x_{i,j-2}$,
$c'=x_{j-1,i}$, and $d_0=x_{j-2,i}$.
Mutation at $a$ commutes with $\mathbf Q_{\ell-2,i}$, since every
vertex in the latter has coordinate sum at most $\ell-3$.  Apply the smaller word first, retaining the columns indexed by its
former boundary vertices, which are not mutated by that word.  The four further mutations
$a,b_0,c',d_0$ give in the two-matrix case
\[
 r_F:
 -e_a+e_{b_0}+e_{F_-}
 \longmapsto -e_{b_0}+e_{d_0}+e_{F_-}
 \longmapsto -e_{b_0}+e_{d_0}+e_{F_-}
 \longmapsto -e_{d_0}+e_H+e_{F_-},
\]
and $r_H=e_{d_0}-e_F-e_{H_+}$.  In the shared-boundary case, the same four mutations replace
$(\alpha,\beta)$ by $(b_0,d_0)$ in all four identities in
\eqref{eq:optimizer-open-row}.  This proves the triangular identities
inductively.  When side labels coincide, their coefficients are added.

It remains to apply $\mathbf R$ in $y$.  Substituting the initial values
of $r_A,r_B$ into \eqref{eq:optimizer-open-row} gives
\begin{align*}
 r_A&=e_{y_{i,j-2}}-e_{y_{i-1,j-1}}-e_{y_{i+1,j-2}}
          +e_\alpha-e_F,\\
 r_B&=e_{y_{i-1,j-1}}+e_{y_{i-2,j+1}}-e_{y_{i-2,j}}
          -e_\beta+e_H\quad(i\ge2).
\end{align*}
In particular the $A$--$B$ arrow has canceled.  Since $j>i$, no other arrows join the vertices of the first two sequences in
\eqref{eq:optimizer-R} to those of the last sequence.  For the sequence $X_t=y_{i,t}$, put
$v=y_{i-1,1}$.  After the initial base mutation, the row at each vertex immediately
before its mutation is
\[
 r_{X_t}=e_v-e_{X_{t-1}}+e_{y_{i-1,t+1}}
                 -e_{X_{t+1}}+e_{y_{i+1,t}}
 \quad(1\le t\le j-2),\qquad
 r_A=e_v-e_{X_{j-2}}+e_\alpha-e_F.
\]
Thus $r_F=e_v-e_A+e_{F_-}$.  For $i\ge2$, put
$Y_t=y_{t,i-t}$ and $Z_t=y_{t,j}$.  For $i\ge3$, the other two sequences give the telescoping identities
\begin{align*}
 r_b^{\rm old}+e_{y_{1,i}}-e_{Y_1}
 &\longmapsto r_b^{\rm old}-e_p+e_{Y_1}-e_{Y_2}
 \longmapsto\cdots\longmapsto
 r_b^{\rm old}-e_p+e_{Y_{i-1}}-e_F,\\
 r_d^{\rm old}+e_{Z_1}-e_{y_{1,j-1}}
 &\longmapsto r_d^{\rm old}+e_q-e_{Z_1}+e_{Z_2}
 \longmapsto\cdots\longmapsto
 r_d^{\rm old}+e_q-e_{Z_{i-1}}+e_H.
\end{align*}
The rows at $Y_1$ and $Z_1$, just before their mutations, contain
$e_b-e_p$ and $e_q-e_d$, respectively; these give the displayed corrections.  If $i=2$
each sequence has one vertex, with rows
$e_b-e_p-e_{y_{1,2}}+e_A-e_F$ and
$e_{y_{1,j-1}}+e_q-e_d-e_\beta+e_H$.
The last mutations also give the formulas for $r_F,r_H$ in
\eqref{eq:optimizer-transfer-rows}, with
$\gamma=Y_{i-1}$ and $\delta=Z_{i-1}=B$.
For $i=1$, $p,q$ are deleted corner labels, so $e_p=e_q=0$; the
vertical sequence alone gives the result with
$\gamma=A$, $\delta=\beta$.

Finally, the new triangles meet the old chain only along its boundary,
which $\mathbf Q^{(m)}\mathbf R^{(m-1)}$ does not mutate.  No old mutable
vertex is adjacent to a vertex in this sequence.  The mutation formula
preserves this nonadjacency, so every old mutable row is unchanged.
\end{proof}

\subsection{The induction and determinant columns}

\begin{proposition}\label{prop:boundary-singleton-isolation}
Every frozen column can be made $e_u$ or $-e_u$ for a mutable index $u$.
For the outer pair $F=[m;i,j]$, $H=[m;j,i]$, the word
$\mathbf T_{\ell,m}\mathbf q_{\ell,m,i}$, where
\begin{equation}
 \mathbf q_{\ell,2,i}=\mathbf Q_{\ell,i}^{(2)},\qquad
 \mathbf q_{\ell,m,i}
 =\mathbf Q_{\ell,i}^{(m)}\mathbf R_{\ell,i}^{(m-1)}
   \mathbf q_{\ell,m-2,i}\,([m-2;i,j],[m-2;j,i]),
 \label{eq:optimizer-recursion}
\end{equation}
gives $B_{\bullet F}=e_u$ and $B_{\bullet H}=-e_u$.
For $m\ge4$, $u=[m-2;j,i]$.
\end{proposition}

\begin{proof}
We induct on even $m$ using the identities
$r_F=-e_u+e_H+e_{F_-}$ and $r_H=e_u-e_F-e_{H_+}$.
The base case is \eqref{eq:optimizer-closed-row}.  After the transfer,
the old mutation sequence uses only old mutable vertices.  Their rows
evolve as in the smaller seed, with zero entries at new coordinates.
Thus
\[
 r_b^{\rm old}\longmapsto-e_w+e_d+e_p,\qquad
 r_d^{\rm old}\longmapsto e_w-e_b-e_q.
\]
The differences from the old rows in \eqref{eq:optimizer-transfer-rows}
are supported on coordinates indexed by old boundary vertices or new
vertices, none of which occurs in the old mutation sequence.  These differences remain unchanged because the old
mutable rows vanish at all new coordinates.
The corrections $-e_p,+e_q$ therefore cancel, leaving
\[
 \begin{aligned}
 r_F&=e_b-e_\gamma+e_{F_-},&r_H&=e_\delta-e_d-e_{H_+},\\
 r_b&=-e_w+e_d+e_\gamma-e_F,&r_d&=e_w-e_b-e_\delta+e_H.
 \end{aligned}
\]
Mutation at $b$, then at $d$, gives
$r_F=-e_d+e_H+e_{F_-}$ and $r_H=e_d-e_F-e_{H_+}$.
Their only nonzero entry at a mutable index is at $d$, and the
symmetrizer entries at $d,F,H$ are all two.  Thus the frozen columns are $e_d,-e_d$.
Every vertex in \eqref{eq:optimizer-recursion} is mutable: the right
boundary of the old chain becomes internal after two triangles are added.

For determinants, return to the rectangular exchange matrix with all
frozen columns retained.  The same twist
and corner identifications give
\begin{equation}
 B_{\bullet\partial_1}=-e_{[2;0,\ell-1]},\qquad
 B_{\bullet\partial_m}=e_{[m;\ell-1,0]},\qquad
 B_{\bullet\partial_r}=e_{[r;\ell-1,0]}-e_{[r;\ell-1,1]}
 \quad(1<r<m).
 \label{eq:optimizer-determinants}
\end{equation}
For an intermediate $r$, mutate the base vertices
$x_k=[r;k,0]$ in order $k=\ell-1,\ldots,1$.
The base-row formula first gives $e_{x_{\ell-2}}-e_{x_{\ell-1}}$;
subsequently the selected column changes by
$e_{x_k}-e_{x_{k+1}}\mapsto e_{x_{k-1}}-e_{x_k}$.
At $k=1$ this is $-e_{x_1}$, since $x_0$ is absent.  Thus every remaining frozen column can be made $e_u$ or $-e_u$.
\end{proof}

\bibliographystyle{amsalpha}
\bibliography{bosonic_plethysm}

\end{document}